\documentclass[a4paper, 12pt, reqno]{amsart}
\usepackage{amsmath,amsthm,amssymb,latexsym,mathrsfs,stmaryrd,color,mathtools}
\usepackage[all]{xy}
\usepackage{url}
\usepackage[a4paper, margin=1.2in]{geometry}
\usepackage{enumerate}

\usepackage[utf8]{inputenc}
\usepackage[T1]{fontenc}

\usepackage{relsize}
\usepackage[bbgreekl]{mathbbol}
\usepackage{amsfonts}

\DeclareSymbolFontAlphabet{\mathbb}{AMSb}
\DeclareSymbolFontAlphabet{\mathbbl}{bbold}

\usepackage{enumitem}
\usepackage[colorlinks=true, hyperindex, linkcolor=magenta, pagebackref=false, citecolor=cyan, pdfpagelabels]{hyperref}
\usepackage[capitalize]{cleveref}
\setlist[enumerate]{itemsep=2pt,parsep=2pt,before={\parskip=2pt}}

\newcommand{\cosimp}[3]{\xymatrix@1{#1 \ar@<.4ex>[r] \ar@<-.4ex>[r] & {\ }#2 \ar@<0.8ex>[r] \ar[r] \ar@<-.8ex>[r] & {\ } #3 \ar@<1.2ex>[r] \ar@<.4ex>[r] \ar@<-.4ex>[r] \ar@<-1.2ex>[r] & \cdots }}

\newcommand{\adjunction}[4]{\xymatrix@1{#1{\ } \ar@<0.3ex>[r]^{ {\scriptstyle #2}} & {\ } #3 \ar@<0.3ex>[l]^{ {\scriptstyle #4}}}}

\numberwithin{equation}{section}

\DeclareMathOperator{\crys}{crys}

\DeclareMathOperator{\Hom}{Hom}

\DeclareMathOperator{\Spf}{Spf}

\DeclareMathOperator{\Spec}{Spec}
\DeclareMathOperator{\ad}{ad}
\DeclareMathOperator{\cont}{cont}

\DeclareMathOperator{\Fil}{Fil}

\DeclareMathOperator{\GL}{GL}
\DeclareMathOperator{\GSp}{GSp}
\DeclareMathOperator{\rig}{rig}
\DeclareMathOperator{\Sp}{Sp}
\DeclareMathOperator{\tri}{tri}
\DeclareMathOperator{\reg}{reg}
\DeclareMathOperator{\Rep}{Rep}
\DeclareMathOperator{\Bun}{Bun}
\DeclareMathOperator{\sat}{sat}
\DeclareMathOperator{\Sym}{Sym}
\DeclareMathOperator{\rank}{rank}
\DeclareMathOperator{\unr}{unr}
\DeclareMathOperator{\dom}{dom}
\DeclareMathOperator{\Lie}{Lie}
\DeclareMathOperator{\Aut}{Aut}
\DeclareMathOperator{\End}{End}
\DeclareMathOperator{\Der}{Der}
\DeclareMathOperator{\Set}{Set}
\DeclareMathOperator{\std}{std}
\DeclareMathOperator{\diag}{diag}
\DeclareMathOperator{\dR}{dR}
\DeclareMathOperator{\pdR}{pdR}
\DeclareMathOperator{\Res}{Res}
\DeclareMathOperator{\sing}{sing}
\DeclareMathOperator{\rec}{rec}
\DeclareMathOperator{\Gal}{Gal}
\DeclareMathOperator{\gr}{gr}
\DeclareMathOperator{\Ext}{Ext}
\DeclareMathOperator{\ob}{ob}
\DeclareMathOperator{\val}{val}

\newtheorem{theorem}{Theorem}[section]
\newtheorem*{theorem*}{Theorem}
\newtheorem*{definition*}{Definition}
\newtheorem{proposition}[theorem]{Proposition}
\newtheorem{lemma}[theorem]{Lemma}
\newtheorem{corollary}[theorem]{Corollary}

\theoremstyle{definition}
\newtheorem{definition}[theorem]{Definition}

\newtheorem{remark}[theorem]{Remark}

\newtheorem{example}[theorem]{Example}

\crefname{assumption}{assumption}{assumptions}
\crefname{construction}{construction}{constructions}

\title[Zariski density of crystalline points for $\GSp_{2n}$]{Zariski density of crystalline points on $\GSp_{2n}$-valued local deformation rings}
\author{Kensuke Aoki}
\address{Department of Mathematics, Faculty of Science, Kyoto University
Kyoto, 606-8502, Japan}
\email{aoki.kensuke.88s@st.kyoto-u.ac.jp}

\begin{document}

\begin{abstract} 
We introduce trianguline deformation spaces $X_{\tri} (\overline{\rho})$ for a $\GSp_{2n}$-valued residual representation $\overline{\rho} \colon {\mathcal{G}}_K \to \GSp_{2n} (k)$, where $k$ is a finite field of characteristic $p > 0$ and ${\mathcal{G}}_K$ is the absolute Galois group of a finite extension $K / {\mathbb{Q}}_p$, and study their properties. 
We show Zariski density of the crystalline points on the rigid generic fiber ${\mathfrak{X}}_{\overline{\rho}}$ of the framed deformation space of $\overline{\rho}$ under some conditions. 
\end{abstract}

\maketitle

\tableofcontents

\section{Introduction}
Let $K$ be a finite extension of ${\mathbb{Q}}_p$ whose Galois group is ${\mathcal{G}}_K$, $L$ be a finite extension of $K$ such that $\Hom_{{\mathbb{Q}}_p} (K, L) = [K : {\mathbb{Q}}_p]$ and $k_L$ be the residue field of $L$. 

\begin{theorem}[{Theorem \ref{gsp_density}}]
\label{gsp_density_intro}
Assume that a fixed representation $\overline{\rho} \colon {\mathcal{G}}_K \to \GSp_{2n} (k_L)$ satisfies the following conditions: 
\begin{enumerate}
\item $H^0 ({\mathcal{G}}_K, \ad (\overline{\rho})) \cong k_L$
\item ${\mathfrak{X}}_{\reg\text{-}\crys}$ is non-empty, 
\item $H^0 ({\mathcal{G}}_K, \ad(\overline{\rho}) \otimes \omega) = 0$, 
\item $\zeta_p \notin K^{\ast}$, or $p \nmid 2n$. 
\end{enumerate}
Then the Zariski closure $\overline{{\mathfrak{X}}}_{\reg\text{-}\crys}$ of the subset of the regular crystalline points in the rigid analytic space ${\mathfrak{X}}_{\overline{\rho}}$ associated to the universal framed deformation ring $R_{\overline{\rho}}$ of $\overline{\rho}$ is the total space ${\mathfrak{X}}_{\overline{\rho}}$. 
\end{theorem}

For a fixed representation $\overline{\rho} \colon {\mathcal{G}}_K \to \GL_m (k_L)$, Emerton-Gee \cite{EG19} has shown the existence of a crystalline lift $\rho \colon {\mathcal{G}}_K \to \GL_m (L)$ for general $\overline{\rho}$. 

On the other hand, Zariski density of crystalline point on the rigid analytic space ${\mathfrak{X}}_{\overline{\rho}} = \Sp \left( R_{\overline{\rho}} \left[ \tfrac{1}{p} \right] \right)$ that is the rigid generic fiber of the framed deformation space of $\overline{\rho}$ has been considered. This problem was solved by Chenevier \cite{Che10}, Nakamura \cite{Nak11}, Iyenger \cite{Iye19} and B\"ockle-Iyenger-Pa\v{s}k\={u}nas \cite{BIP22} in the general case. 

\subsection{Outline of the proof}

Fix $\overline{\rho} \colon {\mathcal{G}}_K \to \GSp_{2n} (k_L)$ from here. 
Theorem \ref{gsp_density_intro} is an analog in the case of the universal framed deformation rings for $\GSp_{2n}$-valued representations $\overline{\rho}$ studied by Balaji \cite{Bal12} in the case of $G = \GSp_{2n}$. 

As a weaker form of Theorem \ref{gsp_density_intro}, we will show the following proposition in advance: 

\begin{proposition}[{Proposition \ref{closure_irr_component}}]
\label{closure_irr_component_intro}
The Zariski closure $\overline{{\mathfrak{X}}}_{\reg\text{-}\crys}$
is a union of irreducible components of ${\mathfrak{X}}_{\overline{\rho}}$. 
\end{proposition}

In \cite{Nak11}, Nakamura has constructed the local eigenvarieties for calculating dimension of a smooth point of $\overline{{\mathfrak{X}}}_{\reg\text{-}\crys}$ in each irreducible component $\mathfrak{Z} \subset {\mathfrak{X}}_{\overline{\rho}}$ and proved the assertion analogous to the case of $\GL_m$. However, we can use the method of \cite{BHS17} showing the accumulation property around the benign points on the trianguline deformation spaces instead. 
We also need to show the dimension is the same as that of ${\mathfrak{X}}_{\overline{\rho}}$. We introduce $(\varphi, \Gamma)$-modules with $\GSp_{2n}$-structures over Robba rings ${\mathcal{R}}_L$ and construct trianguline deformation spaces $X_{\tri} (\overline{\rho})$ for $\GSp_{2n}$ and morphisms $X_{\tri} (\overline{\rho}) \to {\mathfrak{X}}_{\overline{\rho}}$. The correspondence of triangulations of $D$ and refinements $D_{\crys} (D)$ for crystalline $(\varphi, \Gamma)$-modules $D$ is generalized as follows:

\begin{proposition}[{Proposition \ref{G_Berger_corr}}]
\label{G_Berger_corr_intro}
Let $D$ be a crystalline frameable $(\varphi, \Gamma)$-module with $\GSp_{2n}$-structure over ${\mathcal{R}}_L$. Assume the Frobenius action on $r_{\ast} (D_{\crys} (D))$ has $m$ distinct eigenvalues in $K_0 \otimes_{{\mathbb{Q}}_p} L$. Then there is a bijection from the set of triangulations of $D$ to that of refinements of $D_{\crys} (D)$. Moreover, if there is a refinement of $D_{\crys} (D)$, there exists exactly $\# W(G, T)$ refinements of $D_{\crys} (D)$. 
\end{proposition}

By Proposition \ref{G_Berger_corr_intro} and the accumulation property of $X_{\tri} (\overline{\rho})$, replacing $L$ by its finite extension, we may take a unique irreducible component ${\mathfrak{Y}}_w$ in $X_{\tri} (\overline{\rho})$ which maps in $\mathfrak{Z}$ for each element $w \in W(\GSp_{2n}, T)$ of the Weyl group of $\GSp_{2n}$. 
Let ${\mathcal{P}}_w$ be the triangulation of $D$ associated to $w$ and $X_D$, (resp.\ $X_{D, {\mathcal{P}}_w}$) be the (trianguline) deformation functor with the tangent space $TX_D$ (resp.\ $TX_{D, {\mathcal{P}}_w}$). Then the maximality of the dimension of each $\mathfrak{Z}$ is reduced to prove the following surjectivity: 

\begin{corollary}[{Corollary \ref{surjectivity_GSp}}]
\label{surjectivity_GSp_intro}
Let $D$ be a crystalline $(\varphi, \Gamma)$-module with $\GSp_{2n}$-structure over ${\mathcal{R}}_L$ such that $\widetilde{D}$ is $\varphi$-generic with regular Hodge-Tate type and assume that there exists a refinement of $D_{\crys} (D)$. Then the map 
\[
\bigoplus_{w \in W(\GSp_{2n}, T)} TX_{D, {\mathcal{P}}_w} \to TX_D
\]
is surjective. 
\end{corollary}

In the $\GL_m$ case, Corollary \ref{surjectivity_GSp_intro} is solved by Chenevier \cite[Theorem 3.19]{Che11} for $K = {\mathbb{Q}}_p$ and noncritical $D$, Nakamura \cite[Theorem 2.62]{Nak10} for general $K$ and noncritical $D$ and Hellmann-Margerin-Schraen \cite[Theorem 3.11]{HMS18} for general conditions. In \cite{Che11} and \cite{Nak10}, the problem is shown by studying the subspace of crystalline deformations and induction on $m$. In \cite{HMS18}, it is shown by studying the deformation functors of trianguline almost de Rham $L \otimes_{{\mathbb{Q}}_p} B_{\dR}^+$-representations and their local models constructed in \cite{BHS17}. 
However, it seems that the direct generalizations of these methods do not work in the case of $\GSp_{2n}$. We reprove the $\GL_m$-analog \cite[Theorem 3.11]{HMS18} and show Corollary \ref{surjectivity_GSp_intro} explicitly calculating a basis corresponds to each $w \in W(\GSp_{2n}, T)$ in section \ref{explicit_calculation}. 

\subsection*{Acknowledgement}
The author thanks his advisor Tetsushi Ito. He also thanks Kentaro Nakamura for valuable comments on the draft of the paper.

\section{Notation}

For a group $G$ and a subset $S \subset G$, we write $N_G (S)$ for the normalizer of $S$ in $G$ and $C_G (S)$ for the centralizer of $S$ in $G$. 


Let $K$ be a finite extension of ${\mathbb{Q}}_p$ and $K_0$ be the maximal unramified subextension of ${\mathbb{Q}}_p$ in $K$. We write $\zeta_m$ for a $m$-th roots of unity. Let $K_n \coloneqq K(\zeta_{p^n})$ and $K_{\infty} \coloneqq \bigcup_n K_n$. Let $k$ be the residue field of $K$, $k'$ be the residue field $K_{\infty}$ and $K'_0 \coloneqq W(k')$. We write $\Gamma \coloneqq \Gal (K_{\infty} / K)$. 
Let $L$ be a sufficiently large finite extension of $K$ such that $K$ splits in $L$ and $k_L$ be the residue field of $L$ and $\Sigma \coloneqq \Hom_{{\mathbb{Q}}_p} (K, L)$. 
Let ${\mathcal{G}}_K$ be the absolute Galois group of $K$. 
Let $\epsilon \colon {\mathcal{G}}_K \to {\mathbb{Z}}_p^{\ast}$ be the $p$-adic cyclotomic character. 
We fix a uniformizer $\varpi_K \in {\mathcal{O}}_K$ and normalize the reciprocity isomorphism $\rec_K \colon K^{\ast} \to W_K^{\text{ab}}$ of local class field theory such that $\varpi_K$ is mapped to the geometric Frobenius automorphism. 
Let $B_m$, $U_m$, $T_m$ be respectively the standard upper Borel subgroup, its unipotent radical and the split maximal torus of $\GL_m$ for $m \geq 1$. 

If there is no description, $G$ is assumed to be a connected algebraic group over $K$.  
If $G$ is a split algebraic group and $T$ is a split maximal torus of $G$, let $W(G, T) \coloneqq N_G (T) / C_G (T)$ be the Weyl group. 

If $G$ is a group scheme over a ring $A$, let $\Rep_A (G)$ denote the full subcategory of the category of $G_A$-modules whose objects are finite free $A$-modules (cf.\ \cite[\S I.2.7]{Jan07}). 
We write $I_m = \diag (1, \ldots, 1)$ for the identity matrix of size $m$. 
Let ${\mathfrak{S}}_m$ be the symmetric group of degree $m$ for $m \in {\mathbb{Z}}_{\geq 1}$. We have the natural isomorphism ${\mathfrak{S}}_m \cong W(\GL_m, T_m)$. For an $(m \times m)$-matrix $(a_{ij})_{i, j} \in M_m (A)$, we write $^t(a_{ij})_{i, j}$ for the transpose of $(a_{ij})_{i, j}$. 
For a scheme or a rigid space $X$ and its point $x \in X$, let $k(x)$ be the residue field of $x$. 
We define the sign of Hodge-Tate weights of de Rham representations such that the Hodge-Tate weight of the $p$-adic cyclotomic character is $+1$. 

Let ${\mathcal{C}}_L$ (resp.\ ${\mathcal{C}}_{{\mathcal{O}}_L}$) be the category of local artinian $L$-algebra (resp.\ local artinian ${\mathcal{O}}_L$-algebra) with residue field isomorphic to $L$ (resp.\ $k_L$), and for an $A$ in $\text{ob} ({\mathcal{C}}_L)$ (resp, $\text{ob} ({\mathcal{C}}_{{\mathcal{O}}_L}$)) let ${\mathfrak{m}}_A$ denote its maximal ideal and $\phi_A \colon A \to L$ (resp.\ $\phi_A \colon A \to k_L$) denote the natural reduction map. If $A$ is an element of ${\mathcal{C}}_L$ or ${\mathcal{C}}_{{\mathcal{O}}_L}$, we write ${\mathfrak{m}}_A$ for the maximal ideal of $A$. 
For a functor $F \colon {\mathcal{C}}_L \to \Set$, let $TF$ denote the tangent space $\ker (F(L[\epsilon]) \to F(L))$ of $F$. 



For a scheme or a rigid space $X$, we write $\Bun (X)$ for the category of locally free sheaves on the structure sheaf ${\mathcal{O}}_X$. For a sheaf of rings ${\mathcal{A}}_X$ over ${\mathcal{O}}_X$, we define $\Bun ({\mathcal{A}}_X)$ as the category of locally free ${\mathcal{A}}_X$-modules. 
For a commutative ring $R$, let $R[\epsilon]$ be the $R$-algebra naturally isomorphic to $R[T]/(T^2)$ via $\epsilon \mapsto T$. We write $M[\epsilon] \coloneqq M \otimes_R R[\epsilon]$ for the $R[\epsilon]$-module base changed from $M$ by the inclusion $R \to R[\epsilon]$. For a scheme $X$ over a commutative ring $R$, let $\Lie_R (X) \coloneqq \ker (X(R[\epsilon]) \to X(R))$ be the Lie module over $R$. 
Throughout the paper, a ``torsor" means a torsor in the \'etale topology over a scheme or a rigid space. 

Let $\GSp_{2n}$ be the split reductive group over $\mathbb{Z}$ defined by 
\[
\GSp_{2n} (R) = \{ g \in \GL_{2n} (R) \ | \ {^t}gJg = \text{sim} (g) J \ \text{for some} \ \text{sim} (g) \in R^{\ast} \}
\]
for any commutative ring $R$, where $J$ is defined by $\mathbf{x} = (x_{ij})_{i, j}$ in $\GL_{2n}$ such that $x_{ij} = (-1)^{\text{sgn} (j - i)}$ if $i + j = 2n + 1$ and otherwise $x_{ij} = 0$. There is the natural faithful algebraic representation $r_{\std} \colon \GSp_{2n} \to \GL_{2n}$ of the group scheme over $\mathbb{Z}$ by the definition. 
We define the standard Borel subgroup as $B = \GSp_{2n} \cap B_{2n}$ and the maximal split torus $T = \GSp_{2n} \cap T_{2n}$ of $\GSp_{2n}$, embedding $\GSp_{2n}$ into $\GL_{2n}$ by $r_{\std}$. 
We fix the isomorphism $T \cong {\mathbb{G}}_m^{n+1}$ such that $(t_1, \ldots, t_{n+1}) \in {\mathbb{G}}_m^{n+1} \cong T$ is sent to the element 
\[
(t_1, \ldots, t_n, t_n^{-1} t_{n+1}, \ldots, t_1^{-1} t_{n+1}) \in T_{2n}
\]
by the embedding $T \to T_{2n}$.

\section{$(\varphi, \Gamma)$-modules over Robba rings}

We introduce the \emph{sheaf of relative Robba rings} ${\mathcal{R}}_X$ for a rigid analytic spaces $X$ over ${\mathbb{Q}}_p$. See also \cite[\S 2, \S 6]{KPX12} and \cite[\S 2]{HS13}. 
Let $X$ denote a rigid space over ${\mathbb{Q}}_p$. We write ${\mathbb{U}}_L$ for the unit open disc over a $p$-adic field $L$ and ${\mathbb{U}}_{L, r} \subset {\mathbb{U}}_L$ denotes the admissible open subspace of points of absolute value $\geq r$ for some $r \in p^{\mathbb{Q}} \cap [0, 1)$. 
For an $r$ as above, we write ${\mathcal{R}}_X^r$ for the sheaf 
\[
X \supset U \mapsto \Gamma (U \times {\mathbb{U}}_{r, K'_0}, {\mathcal{O}}_{U \times {\mathbb{U}}_{r, K'_0}}). 
\]
and we write ${\mathcal{R}}_X$ for the sheaf $\bigcup_{r > 0} {\mathcal{R}}_X^r$ of functions converging on the product $X \times {\mathbb{U}}_{K'_0}$ as in \cite[\S 2.2]{Hel12}. 
This sheaf of rings is endowed with a continuous ${\mathcal{O}}_X$-linear ring homomorphism $\varphi \colon {\mathcal{R}}_X \to {\mathcal{R}}_X$ and a continuous ${\mathcal{O}}_X$-linear action of the group $\Gamma$.

If $X = \Sp(A)$ for an affinoid algebra over ${\mathbb{Q}}_p$, we write ${\mathcal{R}}_A \coloneqq \Gamma ({\mathcal{R}}_{\Sp(A)}, {\mathcal{O}}_{{\mathcal{R}}_{\Sp(A)}})$ for the Robba ring with coefficient $A$. 
Note that ${\mathcal{R}}_A = \mathcal{R} \widehat{\otimes}_{{\mathbb{Q}}_p} A$. 
The ring ${\mathcal{R}}_L$ 
is not noetherian, but a B\'ezout domain. Let $M \subset {\mathcal{R}}_L^n$ be a ${\mathcal{R}}_L$-submodule, the module 
\[
M^{\sat} \coloneqq \{m \in {\mathcal{R}}_L^n \ | \ \exists f \in {\mathcal{R}}_L \backslash \{ 0 \}, fm \in M \} = (M \otimes_{{\mathcal{R}}_L} \text{Frac} ({\mathcal{R}}_L)) \cap {\mathcal{R}}_L^n
\]
is called the \emph{saturation} of $M$ and we say that $M$ is \emph{saturated} if $M^{\sat} = M$, or which is the same if ${\mathcal{R}}_L^n / M$ is torsion free (cf.\ \cite[\S 2.2.3]{BC09}).

Let $X$ be a scheme or a rigid space over ${\mathbb{Q}}_p$. A \emph{$K$-filtered $\varphi$-module over $X$} or \emph{a filtered $\varphi$-module over $X$} if the $p$-adic field $K$ is obvious is a coherent ${\mathcal{O}}_X \otimes_{{\mathbb{Q}}_p} K_0$-module $\mathcal{F}$ that is locally on $X$ finite free together with an $\text{id} \otimes \varphi$-linear automorphism $\Phi \colon \mathcal{F} \to \mathcal{F}$ and a filtration 
on ${\mathcal{F}}_K \coloneqq \mathcal{F} \otimes_{K_0} K$ by ${\mathcal{O}}_X \otimes_{{\mathbb{Q}}_p} K$-submodules that are locally on $X$ direct summands as ${\mathcal{O}}_X$-modules. 

A \emph{$(\varphi, \Gamma)$-module} over ${\mathcal{R}}_X$ consists of an ${\mathcal{R}}_X$-module $D$ that is locally on $X$ finite free over ${\mathcal{R}}_X$ together with a $\varphi$-linear isomorphism $\phi \colon D \to D$ and a semi-linear $\Gamma$-action commuting with $\phi$. We write $\Phi F_K (X)$ (resp.\ $\Phi F_{K, m} (X)$) for the category of $K$-filtered $\varphi$-modules (resp.\ $K$-filtered $\varphi$-modules of rank $m$) over $X$. We also write $\Phi \Gamma ({\mathcal{R}}_X)$ (resp.\ $\Phi \Gamma_m ({\mathcal{R}}_X)$) for the category of $(\varphi, \Gamma)$-modules (resp.\ $(\varphi, \Gamma)$-modules of rank $m$) over ${\mathcal{R}}_X$. 

Let $X$ be a rigid space defined over $L$. We associates a $K$-filtered $\varphi$-module $\mathcal{F}$ over $X$ of rank $1$ with a $(\varphi, \Gamma)$-module ${\mathcal{R}}_X (\mathcal{F})$ over $X$ as follows. First we assume $X$ is an affinoid $\Sp (A)$. By \cite[Lemma 6.2.3]{KPX12}, there exists a uniquely determined 
\[
(a, (k_{\tau})_{\tau}) \in \Gamma (X, {\mathcal{O}}_X^{\ast}) \times {\mathbb{Z}}^{\Sigma}
\]
such that 
\[
\mathcal{F} \cong \mathcal{F} (a ; (k_{\tau})_{\tau}) \coloneqq {\mathcal{F}}_a \otimes_{K_0 \otimes_{{\mathbb{Q}}_p} {\mathcal{O}}_X} \mathcal{F} ((k_{\tau})_{\tau})
\]
such that ${\mathcal{F}}_a$ is a unique free rank one ${\mathcal{O}}_X \otimes_{{\mathbb{Q}}_p} K_0$-module on which $\Phi^{[K_0 \colon {\mathbb{Q}}_p]}$ act via the multiplication with $a \otimes 1$, the graded module of ${\mathcal{F}}_{a, K}$ is concentrated on degree $0$, the underlying $\varphi$-module structure of ${\mathcal{F}} ((k_{\tau})_{\tau})$ is trivial, and 
\[
(\gr^i (\mathcal{F}((k_{\tau})_{\tau})) \otimes_{{\mathcal{O}}_X \otimes_{{\mathbb{Q}}_p} K, \text{id} \otimes \tau} {\mathcal{O}}_X \cong \begin{cases} 0 & (i \neq -k_{\tau}) \\ {\mathcal{O}}_X & (i = -k_{\tau}) \end{cases}
\]
for all $\tau \in \Sigma = \Hom_{{\mathbb{Q}}_p} (K, L)$. If $\mathcal{F} = {\mathcal{F}}_a$ for some $a \in \Gamma (X, {\mathcal{O}}_X^{\ast})$, we define 
\[
\mathcal{R}_X ({\mathcal{F}}_a) \coloneqq {\mathcal{F}}_a \otimes_{{\mathcal{O}}_X \otimes_{{\mathbb{Q}}_p} K_0} {\mathcal{R}}_X
\]
where $\varphi$ acts diagonally and $\Gamma$ acts on the second factor. Let $\delta \colon K^{\ast} \to (\Gamma (X, {\mathcal{O}}_X))^{\ast}$ be a continuous character. We can write $\delta = \delta_1 \delta_2$ such that $\delta_1 |_{{\mathcal{O}}_K^{\ast}} = 1$ and $\delta_2$ extends to a character of ${\mathcal{G}}_K$. Then we set 
\[
{\mathcal{R}}_X (\delta) \coloneqq {\mathcal{R}}_X ({\mathcal{F}}_{\delta_1 (\varpi_K)}) \otimes_{{\mathcal{R}}_X} {\mathcal{F}}_{\rig} (\delta_2). 
\]
The construction of ${\mathcal{R}}_X (\delta)$ is well-defined and the multiplications of characters $\delta \colon K^{\ast} \to \Gamma (X, {\mathcal{O}}_X)$ correspond to tensor products of $(\varphi, \Gamma)$-modules. 

For an embedding $\tau \in \Sigma$, we also write $\tau \colon K^{\ast} \to L^{\ast}$ for the restriction of $\tau$ to $K^{\ast}$. It follows that ${\mathcal{R}}_L (\tau)$ is a $(\varphi, \Gamma)$-submodule of the trivial one ${\mathcal{R}}_L$ by \cite[Example 6.2.6, (2)]{KPX12}. Since ${\mathcal{R}}_L$ is a B\'ezout domain, there exists an element $t_{\tau} \in {\mathcal{R}}_L$ generating the submodule ${\mathcal{R}}_L (\tau)$. Using the notations, we write 
\[
{\mathcal{R}}_X (\mathcal{F} ((k_{\tau})_{\tau})) \coloneqq \left( \prod_{\tau \in \Sigma} t_{\tau}^{k_{\tau}} \right) \cdot {\mathcal{R}}_X \subset {\mathcal{R}}_X \left[ \tfrac{1}{t} \right]
\]
with actions of $\varphi$ and $\Gamma$ are inherited from ${\mathcal{R}}_X \left[ \tfrac{1}{t} \right]$. Finally we define 
\[
{\mathcal{R}}_X (\mathcal{F}) \coloneqq {\mathcal{R}}_X ({\mathcal{F}}_a) {\otimes}_{{\mathcal{R}}_X} {\mathcal{R}}_X (\mathcal{F} ((k_{\tau})_{\tau})). 
\]
For a general $X$, the constructions $\delta \mapsto {\mathcal{R}}_X (\delta)$ and $\mathcal{F} \mapsto {\mathcal{R}}_X (\mathcal{F})$ are valid by gluing locally on $X$. In the case $X = \Sp (L)$, we note that the construction is compatible with the $\delta \mapsto {\mathcal{R}}_X (\delta)$ via the functors $D_{\crys}$ and $V_{\crys}$ of Fontaine. 

The functor $\Hom_{\cont} ({\mathcal{O}}_K^{\ast}, {\mathbb{G}}_m (-))$ from the category of rigid spaces over ${\mathbb{Q}}_p$ to the category of abelian groups is representable by the rigid generic fiber of $\Spf ({\mathbb{Z}}_p [[{\mathcal{O}}_K^{\ast}]])$. Let $\mathcal{W}$ denote the rigid space and it is called the \emph{weight space} of $K$. We define ${\mathcal{T}}$ as the functor that associate a rigid space $X$ with the abelian group $\Hom_{\cont} (K^{\ast}, {\mathbb{G}}_m (X))$. It is also representable by a rigid space and write $\mathcal{T}$ for it. There is a natural restriction $\mathcal{T} \to \mathcal{W}$ and $\mathcal{T}$ is identified with ${\mathbb{G}}_m \times \mathcal{W}$ by $\delta \to (\delta(\varpi_K), \delta |_{{\mathcal{O}}_K^{\ast}})$. 
Note that the set of isomorphism classes of $(\varphi, \Gamma)$-modules over ${\mathcal{R}}_L$ is naturally bijective to $\mathcal{T} (L)$ by \cite[Theorem 6.2.14, (i)]{KPX12}. 

We define a Herr complex $C^{\bullet} (D)$ as 
\[
D^{\Delta} \xrightarrow{(\varphi - 1, \gamma' - 1)} D^{\Delta} \oplus D^{\Delta} \xrightarrow{(1 - \gamma') \oplus (\varphi - 1)} D^{\Delta}
\]
where $\gamma'$ is a topological generator of $\Gamma / \Delta$ and write $H^i (D)$ for $i$-th cohomology of the complex $C^{\bullet} (D)$ for $i \in [0, 2]$. 
We also write $\epsilon$ for the  composition $\epsilon \circ \rec_K \colon K^{\ast} \to L^{\ast}$ of the $p$-adic cyclotomic character and $\rec_K$, which is the same as $\delta (1, \ldots, 1) |\delta (1, \ldots, 1)|$. 

\begin{lemma}
\label{phiGamma_L_coh}
For a $\delta \in \mathcal{T} (L)$, then 
\begin{enumerate}
\item $H^0 ({\mathcal{R}}_L (\delta)) \neq 0$ if and only if $\delta = \delta((-k_{\tau})_{\tau})$ for some $(k_{\tau})_{\tau} \in {\prod}_{\tau \colon K \hookrightarrow L} {\mathbb{Z}}_{\geq 0}$. 
\item $H^2 ({\mathcal{R}}_L (\delta)) \neq 0$ if and only if $\delta = \epsilon \cdot \delta((k_{\tau})_{\tau})$ for some $(k_{\tau})_{\tau} \in {\prod}_{\tau \colon K \hookrightarrow L} {\mathbb{Z}}_{\geq 0}$. 
\item $\dim_L (H^1 ({\mathcal{R}}_L (\delta)) = [K \colon {\mathbb{Q}}_p]$ if and only if 
$\delta \neq \delta ((-k_{\tau})_{\tau})$ and $\delta \neq \epsilon \cdot \delta ((k_{\tau})_{\tau})$ for all $(k_{\tau})_{\tau} \in {\prod}_{\tau} {\mathbb{Z}}_{\geq 0}$.
\end{enumerate}
\end{lemma}

\begin{proof}
See \cite[Lemma 2.1]{HS13} (cf.\ also \cite[Proposition 3.1]{Col08}, \cite[Proposition 2.14]{Nak08} and \cite[Proposition 6.2.8]{KPX12}). 
\end{proof}

\begin{corollary}
\label{subphiGamma_rank1}
Let $\delta \colon K^{\ast} \to L^{\ast}$ be a continuous character. Then a non-zero $(\varphi, \Gamma)$-submodule of ${\mathcal{R}}_L (\delta)$ is of the form $({\prod}_{\tau \in \Sigma} t_{\tau}^{k_{\tau}}) {\mathcal{R}}_L (\delta)$ where $k_{\tau} \geq 0$. 
\end{corollary}

\begin{proof}
This follows from Lemma \ref{phiGamma_L_coh}.
\end{proof}

We write ${\mathcal{T}}_{\reg} \subset \mathcal{T}$ for the set of \emph{regular characters}, which is the set of $\delta \in \mathcal{T}$ such that $\delta \neq \delta ((-k_{\tau})_{\tau})$ and $\delta \neq \epsilon \cdot \delta ((k_{\tau})_{\tau})$ for all $(k_{\tau})_{\tau} \in {\prod}_{\tau} {\mathbb{Z}}_{\geq 0}$.

For an integer $d > 0$, let ${\mathcal{T}}_{\reg}^d \subset {\mathcal{T}}^d$ denote the set of \emph{regular parameters} $(\delta_1, \ldots, \delta_d) \in {\mathcal{T}}^d$ such that $\delta_i / \delta_j \in \mathcal{T}_{\reg}$ for $i < j$. Note that $\mathcal{T}_{\reg}^d \neq (\mathcal{T}_{\reg})^d$.

A character $\delta \in \mathcal{T} (A)$ is \emph{algebraic of weight $(k_{\tau})_{\tau}$} if $\delta = \delta ((k_{\tau})_{\tau})$ for $(k_{\tau})_{\tau} \in {\mathbb{Z}}^{\Sigma}$. 
We say that $\delta \in \mathcal{T} (A)$ is \emph{locally algebraic of weight $(k_{\tau})_{\tau}$} if $\delta \otimes \delta ((-k_{\tau})_{\tau})$ becomes trivial after restricting to some open subgroup of ${\mathcal{O}}_K^{\ast}$. 

A point $\underline{\delta} = (\delta_1, \ldots, \delta_d) \in {\mathcal{T}}^d (A)$ is \emph{algebraic} (resp.\ \emph{locally algebraic}) if the $i$-th factor $\delta_i$ of $\underline{\delta}$ is algebraic (resp.\ locally algebraic) for $i \in [1, d]$. 
The (locally) algebraicity of elements of ${\mathcal{W}}^d (A)$ is similarly defined. 

\begin{lemma}
\label{phiGamma_ext_sheaf}
Let $\delta = (\delta_1, \ldots, \delta_d) \in ({\mathcal{T}}_{\reg})^d (X)$ for a rigid space $X$ over $L$. If $D$ is a successive extension of the $\mathcal{R}_X (\delta_i)$, then $H^1 (D)$ is a locally free ${\mathcal{O}}_X$-module of rank $d[K \colon {\mathbb{Q}}_p]$. 
\end{lemma}

\begin{proof}
See \cite[Proposition 2.3]{HS13}. 
\end{proof}


\section{$(\varphi, \Gamma)$-modules with $G$-structures}

Let $A$ be an affinoid algebra over ${\mathbb{Q}}_p$ and ${\mathcal{R}}_A$ be the Robba ring with coefficient $A$. We introduce frameable $(\varphi, \Gamma)$-modules with $G$-structures over ${\mathcal{R}}_A$ defined in \cite[\S 9.2]{Lin23a}. 
We fix a faithful algebraic embedding $G \to \GL (V)$ over $K$ and choose a trivialization $b_0 \colon V \xrightarrow{\cong} {\mathbb{G}}_a^{\oplus m}$. 
Write 
$V_{(-)}$ for $V \otimes_K (-)$. 
Let $T^{b_0}$ be  the trivial $G$-torsor over ${\mathcal{R}}_A$ defined by 
\begin{equation}
\label{G_torsor_triv}
R \mapsto \{ \text{trivializations}\ b \colon V_R \xrightarrow{\cong} R^{\oplus m} \ | \ b = g \circ b_0 \ \text{for some} \ g \in G(\mathcal{R}_A) \}.
\end{equation}
Write ${\varphi}^{\ast} V_{\mathcal{R}_A}$ for $\mathcal{R}_A {\otimes}_{\varphi, \mathcal{R}_A} V_{\mathcal{R}_A}$, where $\varphi$ acts on the first factor, and $\varphi^{\ast} b_0 \colon {\varphi}^{\ast} V_{\mathcal{R}_A} \to {\mathcal{R}_A}^{\oplus m}$ for the trivialization which sends $1 \otimes b_0^{-1} (e_i)$ to $e_i$, where $\{ e_1, \ldots , e_m \}$ is the standard basis for ${\mathbb{G}}_a^{\oplus m}$. 
Similarly as above, write $T^{{\varphi}^{\ast} b_0}$ for trivial $G$-torsor over $\mathcal{R}_A$ defined by 
\begin{equation}
\label{G_torsor_triv_phi}
R \mapsto \{ \text{trivializations} \ b \ \colon {\varphi}^{\ast} V_R \xrightarrow{\cong} R^{\oplus m} \ | \ b = g \circ \varphi^{\ast} b_0 \ \text{for some} \ g \in G(\mathcal{R}_A) \}.
\end{equation}
The cases of $T^{{\gamma}^{\ast} b_0}$ for $\gamma \in \Gamma$ are the same as above. Let $f \colon V_{\mathcal{R}_A} \to V_{\mathcal{R}_A}$ be a $\varphi$-semilinear morphism, Under the condition of $b_0 ({\varphi}^{\ast} f) \in T^{{\varphi}^{\ast} b_0}$, the association $b \mapsto b({\varphi}^{\ast} f)$ defines a morphism $T^{b_0} \to T^{{\varphi}^{\ast} b_0}$ of $G$-torsor. Let $f$ also denote the last $G$-torsor morphism and $f_{b_0}$ denote the morphism $G_{\mathcal{R}_A} \to G_{\mathcal{R}_A}$ of trivial $G$-torsors induced by $f$ and the fixed trivializations $b_0$, $\varphi^{\ast} b_0$. 
Let $I_m$ be the matrix $(e_1 | e_2 | \cdots | e_m) \in \GL({\mathbb{G}}_a^m)$ and we put $[f]_{b_0} \coloneqq f_{b_0} (I_m)$. Then we have the following lemma.

\begin{lemma}[{\cite[Lemma 9.2.1, Lemma 9.2.2]{Lin23a}}]
For a $\varphi$-semilinear map $f \colon V_{\mathcal{R}_A} \to V_{\mathcal{R}_A}$, a $\gamma$-semilinear map $g \colon V_{\mathcal{R}_A} \to V_{\mathcal{R}_A}$ for $\gamma \in \Gamma$ and $h \in G(\mathcal{R}_A)$, 
\begin{align*}
[f \circ g]_{b_0} &= [f]_{b_0} \varphi ([g]_{b_0}), \\ 
[g \circ f]_{b_0} &= [g]_{b_0} \gamma ([f]_{b_0}), \\ 
[f]_{h \circ b_0} &= h[f]_{b_0} \varphi (h)^{-1}.
\end{align*}
\end{lemma}




Let $G$ be an algebraic group over $K$. 

\begin{definition}
We define \emph{framed $(\varphi, \Gamma)$-modules with $G$-structures} over ${\mathcal{R}_A}$ as data $([\phi], [\gamma] \ | \ \gamma \in \Gamma)$ of coordinate matrices of $\GL({\mathcal{R}_A^{\oplus m}})$ such that $[\phi] \varphi ([\gamma]) = [\gamma] \gamma ([\phi])$ for any $\gamma \in \Gamma$ and 
\[
\Gamma \to \GL({\mathcal{R}}_A^{\oplus m}), \quad \gamma \mapsto [\gamma]
\]
are continuous homomorphism. A $G$-torsor with semilinear $(\varphi, \Gamma)$-structure which is isomorphic to a framed $(\varphi, \Gamma)$-modules with $G$-structures is called \emph{frameable $(\varphi, \Gamma)$-modules with $G$-structures}. 
\end{definition}
By the associations (\ref{G_torsor_triv}), (\ref{G_torsor_triv_phi}) and their subsequent argument, it is also possible that framed $(\varphi, \Gamma)$-modules with $G$-structures are constructed as geometric torsors naturally included in $(\varphi, \Gamma)$-modules with $\GL (V)$-structures. \cite[Lemma 9.2.3]{Lin23a} shows that the two pairs $([\phi], [\gamma] \ | \ \gamma \in \Gamma)$ and $([\phi]', [\gamma]' \ | \ \gamma \in \Gamma)$ are equivalent (as frameable $(\varphi, \Gamma)$-modules with $G$-structures) if and only if there exists $h \in G(\mathcal{R}_A)$ such that $[\phi] = h [\phi]' \varphi (h)^{-1}$ and $[\gamma] = h [\gamma]' \gamma (h)^{-1}$ for any $\gamma \in \Gamma$. Let $\Phi \Gamma_G ({\mathcal{R}_A})$ denote the category of frameable $(\varphi, \Gamma)$-modules with $G$-structures over ${\mathcal{R}_A}$. 


We will associate representations $\rho \colon {\mathcal{G}}_K \to \GSp_{2n} (A)$ with framed $(\varphi, \Gamma)$-module with $G$-structure over ${\mathcal{R}_A}$. 

\begin{lemma}
\label{symp_Gal_eqv}
There is an equivalence of groupoids 
\[
\{ \rho \colon {\mathcal{G}}_K \to \GSp_{2n} (A) \} \xrightarrow{\sim} \left\{(\rho', \chi, \alpha) | \begin{array}{l} \rho' \colon {\mathcal{G}}_K \to \GL_{2n} (A), \ \chi \colon {\mathcal{G}}_K \to \GL_1 (A), \\ \alpha \colon \rho' \cong (\rho')^{\lor} \otimes \chi, \ ({\alpha}^{\lor} \otimes \chi)^{-1} \circ \alpha = -1 \end{array} \right\}.
\]
\end{lemma}

\begin{proof}
See \cite[Lemma 4.1.2]{Lee23}. 
\end{proof}

Similarly, we obtain the equivalence for framed $(\varphi, \Gamma)$-modules with $G$-structures. Let $\Phi \Gamma_m ({\mathcal{R}}_A)$ denote the category of classical finite projective $(\varphi, \Gamma)$-modules over ${\mathcal{R}}_A$ which have rank $m$. 

\begin{lemma}
\label{symp_phiGamma_eqv}
There is an equivalence of groupoids 
\[
\Phi \Gamma_{\GSp_{2n}} ({\mathcal{R}}_A) \xrightarrow{\sim} \left\{(\widetilde{D}, \delta, \alpha) | \begin{array}{l} \widetilde{D} \in \ob(\Phi \Gamma_{2n} ({\mathcal{R}}_A)), \ \delta \in \ob(\Phi \Gamma_1 ({\mathcal{R}}_A)), \\ \alpha \colon \widetilde{D} \cong (\widetilde{D})^{\lor} \otimes \delta, \ ({\alpha}^{\lor} \otimes \delta)^{-1} \circ \alpha = -1 \end{array} \right\}.
\]
\end{lemma}

Then we construct a symplectic version of the functor $D_{\rig}$ (cf.\ \cite[Theorem 2.2.17]{KPX12}) as the composition of the following functors: 
\[
(\rho \colon {\mathcal{G}}_K \to \GSp_{2n} (A)) \mapsto (\rho', \chi, \alpha) \xmapsto{D_{\rig}} (D_{\rig} (\rho'), D_{\rig} (\chi), D_{\rig} (\alpha)), 
\]
where the first map is from Lemma \ref{symp_Gal_eqv}. Since $(D_{\rig} (\rho'), D_{\rig} (\chi), D_{\rig} (\alpha))$ is in the essential image of the functor in Lemma \ref{symp_phiGamma_eqv}, it determines functorially the equivalence class of frameable $(\varphi, \Gamma)$-modules with $G$-structures over ${\mathcal{R}}_A$. Fix a quasi-inverse of the equivalence in Lemma \ref{symp_phiGamma_eqv} and let $D_{\rig} (\rho)$ denote the image of $(D_{\rig} (\rho'), D_{\rig} (\chi), D_{\rig} (\alpha))$ by the functor.

\begin{lemma}
There exists a fully faithful and exact functor 
\[
\{ \rho \colon {\mathcal{G}}_K \to \GSp_{2n} (A) \} \to \Phi \Gamma_{\GSp_{2n}} ({\mathcal{R}}_A) \quad (\rho \colon {\mathcal{G}}_K \to \GSp_{2n} (A)) \mapsto D_{\rig} (\rho)
\]
from the category of continuous homomorphisms $\rho \colon {\mathcal{G}}_K \to \GSp_{2n} (A)$ to the category of frameable $(\varphi, \Gamma)$-modules with $G$-structures over ${\mathcal{R}}_A$. It also commutes with base change in $A$.
\end{lemma}

\begin{proof}
Existence of a functor follows by the above argument. It follows that the functor is fully faithful, exact and commuting with base changes by \cite[Theorem 2.2.17]{KPX12}. 
\end{proof}

Let $P$ be a parabolic subgroup of $G$. There exists the Levi decomposition $P = U \rtimes L$ where $U$ is the unipotent radical and $L$ is the Levi subgroup. Write $\ad \colon L \to \GL (U)$ for the adjoint action. We assume that $U$ is abelian and fix an isomorphism $b \colon U \cong {\mathbb{G}}_a^{\oplus m}$. 
When $([\phi], [\gamma] \ | \ \gamma \in \Gamma) \in \Phi \Gamma_P (\mathcal{R})$, write $[\phi] = [\phi_U][\phi_L]$ and $[\gamma] = [\gamma_U][\gamma_L]$ for the factors of Levi decompositions. 
By abuse of notation, write $\ad_{[-]}$ for the matrix $b \circ \ad_{[-]} \circ b^{-1} \in \GL ({\mathbb{G}}_a^{\oplus m})$. 

\begin{lemma}
We have 
\[
(\ad_{[\phi_L]}, \ad_{[\gamma_L]} \ | \ \gamma \in \Gamma) \in \ob (\Phi \Gamma_{\GL ({\mathbb{G}}_a^{\oplus m})} ({\mathcal{R}}_A)). 
\]
\end{lemma}

\begin{proof}
It follows $\ad_{[\phi]} \varphi (\ad_{[\gamma]}) = \ad_{[\gamma]} \gamma (\ad_{[\phi]})$ by \cite[Lemma 9.4.2]{Lin23a}. Since the conjugate actions are continuous homomorphisms, $\ad_{[\gamma]}$ for $\gamma \in \Gamma$ induce a continuous homomorphism $\Gamma \to \GL({\mathcal{R}}_A^{\oplus m})$. 
\end{proof}


For a framed $(\varphi, \Gamma)$-module with $\GL_m$-structure $([\phi], [\gamma] \ | \ \gamma \in \Gamma)$, we define the \emph{framed Herr complex} $C^{\bullet} (\ad_{[\phi_L]}, \ad_{[\gamma_L]} \ | \ \gamma \in \Gamma)$ as 
\[
({\mathcal{R}}_A^{\oplus m})^{\Delta} \xrightarrow{(\varphi - [\phi]^{-1}, \gamma' - [\gamma']^{-1})} ({\mathcal{R}}_A^{\oplus m})^{\Delta} \oplus ({\mathcal{R}}_A^{\oplus m})^{\Delta} \xrightarrow{(\gamma' - \phi([\gamma'])^{-1}, \gamma' ([\phi])^{-1} - \varphi} ({\mathcal{R}}_A^{\oplus m})^{\Delta}
\]
where $\gamma'$ is a topological generator in $\Gamma / \Delta$. This complex is naturally isomorphic to the Herr complex $C^{\bullet} (D)$ where $D$ is a frameable $(\varphi, \Gamma)$-module with $G$-structure that corresponds to $([\phi], [\gamma] \ | \ \gamma \in \Gamma)$. 

\begin{lemma}
For a framed $(\varphi, \Gamma)$-module with $L$-structure $([\phi_L], [\gamma_L] \ | \ \gamma \in \Gamma)$, the $A$-module of the extension classes of $([\phi_L], [\gamma_L] \ | \ \gamma \in \Gamma)$ to a framed $(\varphi, \Gamma)$-module with $P$-structure is naturally isomorphic to the module $H^1 (C^{\bullet} (\ad_{[\phi_L]}, \ad_{[\gamma_L]} \ | \ \gamma \in \Gamma))$. 
\end{lemma}

\begin{proof}
This follows by the argument of \cite[Corollary 10.1.6]{Lin23a}. 
\end{proof}

We introduce the Tannakian formalism for $G$-torsors over $X$. 

\begin{lemma}
\label{torsor_fiber_functor}
For a scheme or a rigid space $X$ over $K$, the category of $G$-torsors over $X$ is naturally equivalent to the category of exact tensor functors $\Rep_K (G) \to \Bun (X)$. 
\end{lemma}

\begin{proof}
The case of schemes is \cite[Theorem 2.5.2]{Lev13}. For rigid spaces, the proof of \cite[Theorem 19.5.2]{SW20} is applied in the case that $\Rep_K (\mathcal{G}) = \Rep_K (G)$ is defined over $K$. Since the set of quasi-compact admissible opens of $X$ is naturally bijective to that of quasi-compact opens of the associated adic space $X^{\ad}$, the Lemma follows. 
\end{proof}

Let $G$ be a split reductive group over $K$ and $X$ be a scheme or a rigid space. We define a \emph{$(\varphi, \Gamma)$-module with $G$-structure} over ${\mathcal{R}}_X$ as a exact tensor functor $\Rep_K (G) \to \Phi \Gamma ({\mathcal{R}}_X)$ whose underlying fiber functor $\Rep_K (G) \to \Bun ({\mathcal{R}}_X)$ is associated with a geometric \'etale $G$-torsor which is \'etale locally on $X$ trivial by the equivalence of Lemma \ref{torsor_fiber_functor}. A $(\varphi, \Gamma)$-module with $G$-structure over ${\mathcal{R}}_X$ can be seen \'etale locally on $X$ frameable by giving some framing. 
For the natural immersion $T \subset G$, fix an isomorphism $T \cong {\mathbb{G}}_m^d$. We write ${\mathcal{T}}_G$ for the parameter space ${\mathcal{T}}^d$ of $(\varphi, \Gamma)$-modules with ${\mathbb{G}}_m^d$-structures. Similarly define ${\mathcal{W}}_G$ and the restriction ${\mathcal{T}}_G \to {\mathcal{W}}_G$. 

We say that an element $\underline{\delta}$ in ${\mathcal{T}}_G (A)$ or ${\mathcal{W}}_G (A)$ for $A \in {\mathcal{C}}_L$ is \emph{algebraic} (resp.\ \emph{locally algebraic}) if the corresponding object in ${\mathcal{T}}^d$ is algebraic (resp.\ locally algebraic). 
Note that for quasi-split reductive groups, the Weyl group scheme $W(G, T)$ is generated by symmetries for each root of $(G, T)$. Thus, for a closed immersion $(G, B, T) \to (G, B, T)$ of quasi-split reductive groups, there is a natural inclusion $W(G, T) \hookrightarrow W(G', T')$ of the Weyl groups. 

\begin{definition}
\label{emb_regular_type}
We define \emph{embeddings of $(G, B, T)$ of regular type} as faithful algebraic embedding $i \colon G \to \GL_m$ such that the following conditions are satisfied: 
\begin{enumerate}
\item $i$ induces a morphism $(G, B, T) \to (\GL_m, B_m, T_m)$ and $G \cap B_m = B$. 
\item The inverse image of ${\mathcal{T}}_{\GL_m, \reg}$ under the map ${\mathcal{T}}_{G, \reg} \to {\mathcal{T}}_{\GL_m, \reg}$ is Zariski open and Zariski dense in ${\mathcal{T}}_{G, \reg}$. 
\item For the sets $\Delta (B) \coloneqq \Delta (G, B, T) = \{ \alpha_1, \ldots, \alpha_l \}$, $\Delta (B_m) \coloneqq \Delta (\GL_m, B_m, T_m)$ of simple roots of $(G, B, T)$, $(\GL_m, B_m, T_m)$, there is a partition $\Delta (B_m) = \Delta_1 \sqcup \cdots \sqcup \Delta_l$ such that the subgroup $W_{\Delta_i} \subset W(\GL_m, T_m)$ associated with $\Delta_i \subset \Delta (B_m)$ contains the nontrivial symmetry $s_i$ in $W_{\{ \alpha_i \}} \subset W(G, T)$. 
\item There exists a sequence $\emptyset = I_0 \subset I_1 \subset \cdots \subset I_s = \Delta (B)$ of subsets of $\Delta (B)$ such that the unipotent radical ${\widetilde{U}}'_{i+1}$ of ${\widetilde{L}}_{i+1} \cap {\widetilde{P}}_i = {\widetilde{L}}_i ({\widetilde{L}}_{i+1} \cap {\widetilde{U}}_i)$ is abelian for $i \in [0, s-1]$, where ${\widetilde{P}}_i$ is the parabolic subgroup of $\GL_m$ associated with the subset $I_i \subset \Delta (B) \subset \Delta (B_m)$ and ${\widetilde{P}}_i = {\widetilde{U}}_i \rtimes {\widetilde{L}}_i$ is the Levi decomposition. 
\end{enumerate}
A split reductive group $(G, B, T)$ is called \emph{of regular type} if it has an embedding of regular type. 
\end{definition}

\begin{example}
Let us consider the case $G = \GSp_{2n}$, $B$ is the standard Borel subgroup and $T$ is the maximal split torus. The standard representation $r_{\std} \colon \GSp_{2n} \to \GL_{2n}$ is regular type. It induces the injection of the Weyl groups $W(\GSp_{2n}, T) \hookrightarrow W(\GL_{2n}, T_{2n}) \cong {\mathfrak{S}}_{2n}$. As the subgroup of ${\mathfrak{S}}_{2n}$, $W(\GSp_{2n}, T)$ is generated by 
\[
\{ (1, 2)(2n-1, 2n), (2, 3)(2n-2, 2n-1), \ldots, (n-1, n)(n+1, n+2), (n, n+1) \}. 
\]
\end{example}


The following lemma is often used in the paper. 

\begin{lemma}
\label{parab_ext_class}
Let $R$ be a commutative ring and $G$ be a reductive group scheme over $R$. For any $R$-Levi subgroup of a parabolic $R$-subgroup of $G$, the natural map 
\[
H^1 (R, L) \to H^1 (R, P)
\]
is bijective. 
\end{lemma}

\begin{proof}
See \cite[XXVI, Corollaire 2.3]{SGA3.3}. 
\end{proof}

\begin{lemma}
\label{phiGamma_parab_ext}
Let $P$ be a parabolic subgroup and $P = U \rtimes L$ be the Levi decomposition. Then any extension of a $(\varphi, \Gamma)$-module with $L$-structure (forgetting the $(\varphi, \Gamma)$-structure) over ${\mathcal{R}}_X$ to a $P$-torsor over $X$ is \'etale locally on $X$ trivial. 
\end{lemma}

\begin{proof}
The assertion follows by that $(\varphi, \Gamma)$-module with $B$-structure over ${\mathcal{R}}_X$ is \'etale locally on $X$ trivial and Lemma \ref{parab_ext_class}. 
\end{proof}

Let $(G, B, T)$ be a split reductive group of regular type and fix an embedding $G \to \GL_m$ of regular type. We write ${\mathcal{T}}_{G, \reg}$ for the inverse image of ${\mathcal{T}}_{\GL_m, \reg}$ under the map ${\mathcal{T}}_G \to {\mathcal{T}}_{\GL_m}$. We set $p_i$ to be the composition of the fixed isomorphism $T \cong {\mathbb{G}}_m^d$ and the $i$-th projection ${\mathbb{G}}_m^d \to {\mathbb{G}}_m$. 

Let $D$ be a $(\varphi, \Gamma)$-module with $G$-structure over ${\mathcal{R}}_X$. A $(\varphi, \Gamma)$-module with $B$-structure $\mathcal{P}$, which is naturally included in $D$ and compatible with the $G$-structure of $D$ and $\mathcal{P} \times^B T \times^{p_i} {\mathbb{G}}_m$ is isomorphic to ${\mathcal{R}}_X (\delta_i)$ for some $\delta_i \in \mathcal{T}(X)$ is called the \emph{triangulation (with the parameter $\delta = (\delta_1, \ldots, \delta_d)$)} of $D$. 
Note that a triangulation of $D$ can be seen as an exact tensor functor $\Rep_K (B) \to \Phi \Gamma ({\mathcal{R}}_X)$ whose composition with the restriction $\Rep_K (G) \to \Rep_K (B)$ is identical to $D$. 



A $(\varphi, \Gamma)$-module with $G$-structure $D$ which admits such a triangulation $\mathcal{P}$ is said to be \emph{trianguline} (with the parameter $\underline{\delta} = (\delta_1, \ldots, \delta_d)$). When $G = \GL_m$ for $m \in {\mathbb{Z}}_{\geq 1}$, the notations are compatible with those of $(\varphi, \Gamma)$-modules of rank $m$. 
We consider the functor ${\mathcal{S}}^{\Box}$ that assigns to a rigid space $X$ over $L$ the isomorphism classes of quadruples $(D, \mathcal{P}, \underline{\delta}, \nu = (\nu_1, \ldots, \nu_d))$, where $D$ is a $(\varphi, \Gamma)$-module with $G$-structure over ${\mathcal{R}}_X$, $\mathcal{P}$ is a triangulation of $D$, $\delta \in {\mathcal{T}}_{G, \reg}$ and 
\[
\nu_i \colon \mathcal{P} \times^B T \times^{p_i} {\mathbb{G}}_m \xrightarrow{\cong} {\mathcal{R}}_X (\delta_i)
\]
is a trivialization. 

\begin{proposition}
\label{reg_parab_ext}
For a split reductive group $(G, B, T)$ and a fixed embedding $G \to \GL_m$ of regular type, 
the functor ${\mathcal{S}}^{\Box}$ is representable by a rigid space and the map ${\mathcal{S}}^{\Box} \to {\mathcal{T}}_{G, \reg}$ is smooth. 
\end{proposition}

\begin{proof}

Since the embedding $G \to \GL_m$ is of regular type, there is a natural inclusion $\Delta (B) \subset \Delta (B_m)$ of the sets of simple roots and a sequence $\emptyset = I_0 \subset I_1 \subset \cdots \subset I_s = \Delta (B)$ such that the unipotent radical ${\widetilde{U}}'_{i+1}$ of ${\widetilde{L}}_{i+1} \cap {\widetilde{P}}_i = {\widetilde{L}}_i ({\widetilde{L}}_{i+1} \cap {\widetilde{U}}_i)$ is abelian for $i \in [0, s-1]$. It also induce the injection of Weyl groups $W(G) \to W(\GL_m)$. Write $P_i$ for the parabolic subgroup of $G$ associated with the subset $I_i \subset \Delta (B)$. We have ${\widetilde{P}}_i \cap G = P_i$ by the Bruhat decomposition. Thus, the unipotent radical $U'_{i+1}$ of $L_{i+1} \cap P_i = L_i (L_{i+1} \cap U_i)$ is also abelian and we have a natural immersion $U'_{i+1} \hookrightarrow {\widetilde{U}}'_{i+1}$. 
It follows that the $(i+1)$ successive parabolic extensions of the torus $T \subset B$ are in the Borel subgroup $B$ so that are in $L_{i+1} \cap B$ for each $i \in [0, s-1]$. 

For each $i \in [0, s-1]$, we write ${\mathcal{S}}_{i+1}^{\Box}$ for isomorphism classes of quadruples $(D, \mathcal{P}, \underline{\delta}, \nu = (\nu_1, \ldots, \nu_d))$, where $D$ is a 
$L_{i+1} \cap B$-torsor over ${\mathcal{R}}_X$, $\mathcal{P}$ is a triangulation of $D$, $\underline{\delta} \in {\mathcal{T}}_{G, \reg}$ and 
\[
\nu_i \colon \mathcal{P} \times^B T \times^{p_i} {\mathbb{G}}_m \xrightarrow{\cong} {\mathcal{R}}_X (\delta_i)
\]
is a trivialization. Note that ${\mathcal{S}}_s^{\Box} = {\mathcal{S}}^{\Box}$. 

We assume that ${\mathcal{S}}_i^{\Box}$ is representable by a rigid space $X_i$ for some $i \in [1, s-1]$. Let ${\widetilde{\ad}}'_{i+1} \colon {\widetilde{L}}'_{i+1} \to \GL ({\widetilde{U}}'_{i+1})$ (resp.\ $\ad'_{i+1} \colon L'_{i+1} \to \GL (U'_{i+1})$) be the adjoint map where ${\widetilde{L}}'_{i+1}$ (resp.\ $L'_{i+1}$) is the Levi subgroup of ${\widetilde{L}}_{i+1} \cap {\widetilde{P}}_i$ (resp.\ $L_{i+1} \cap P_i$). We also write ${\widetilde{D}}_{i+1}$ (resp.\ $D_{i+1}$) for the induced $(\varphi, \Gamma)$-module with $\GL ({\widetilde{U}}'_{i+1})$-structure (resp.\ $(\varphi, \Gamma)$-module with $\GL (U'_{i+1})$-structure) over ${\mathcal{R}}_{X_i}$ from the universal $(\varphi, \Gamma)$-module with $L_i \cap P_{i-1}$-structure $Y_i$ by ${\widetilde{\ad}}'_{i+1}$ (resp.\ ${\ad}'_{i+1}$). As ${\widetilde{U}}'_{i+1}$ and $U'_{i+1}$ is isomorphic to the direct sum of finite copies of ${\mathbb{G}}_a$, we can use the argument of classical $(\varphi, \Gamma)$-modules and their triangulations for ${\widetilde{D}}_{i+1}$ and $D_{i+1}$. Let $r_i$ be an integer such that $U'_{i+1} \cong {\mathbb{G}}_a^{\oplus r_i}$. 

The $(\varphi, \Gamma)$-module ${\widetilde{D}}_{i+1}$ admits a triangulation, which is defined over $K$, with parameter in $({\mathcal{T}}_{\reg})^d$. So considering the intersection of $D_{i+1}$ and the triangulation of ${\widetilde{D}}_{i+1}$, $D_{i+1}$ here also admits a triangulation with parameter in $({\mathcal{T}}_{\reg})^d$. 
It follows that the extension classes of the universal $(\varphi, \Gamma)$-module with $L_i \cap B$-structure $Y_i$ to $(\varphi, \Gamma)$-modules with $L_{i+1} \cap B$-structures are classified by the locally free sheaf $H^1 (D_{i+1})$ of rank $r_i [K \colon {\mathbb{Q}}_p]$ on $X_i$ by Lemma \ref{phiGamma_ext_sheaf} and Lemma \ref{phiGamma_parab_ext}. 
According to the Tate-duality (\cite[Theorem 4.4.5]{KPX12}), we find that also 
\[
{\mathcal{M}}_{X_i} \coloneqq \mathscr{E}xt_{{\mathcal{R}}_{X_i}}^1 (D_{i+1}, {\mathcal{R}}_{X_i})
\]
is also locally free ${\mathcal{O}}_{X_i}$-module of rank $r_i [K \colon {\mathbb{Q}}_p]$. We may put 
\[
{\mathcal{S}}_{i+1}^{\Box} \coloneqq {\underline{\Spec}}_{X_i} (\Sym^{\bullet} {\mathcal{M}}_{X_i}^{\lor}), 
\] 
where $\underline{\Spec}$ is the relative spectrum in the sense of \cite[2.2]{Con06}. It follows that the morphism ${\mathcal{S}}^{\Box} \to {\mathcal{T}}_{G, \reg}$ is smooth from the construction. 
\end{proof}





\section{Trianguline deformation spaces for $\GSp_{2n}$}

We fix a representation $\overline{\rho} \colon {\mathcal{G}}_K \to \GSp_{2n} (k_L)$. Then there exists the universal framed deformation ring $R_{\overline{\rho}}$ which pro-represents the functor from $A \in {\mathcal{C}}_{{\mathcal{O}}_L}$ to the set of continuous representation $\rho_A \colon {\mathcal{G}}_K \to \GSp_{2n} (A)$ such that $\overline{\rho}$ is the composition 
\[
{\mathcal{G}}_K \xrightarrow{\rho_A} \GSp_{2n} (A) \xrightarrow{\phi_A} \GSp_{2n} (k_L)
\]
for example by \cite[Theorem 1.3.6]{Bal12}. We call the local ring $R_{\overline{\rho}}$ \emph{universal framed deformation ring of $\overline{\rho}$}. We write ${\mathfrak{X}}_{\overline{\rho}}$ for the rigid analytic space over $L$, which is the rigid generic fiber of the universal framed deformation space $\Spf (R_{\overline{\rho}})$ of $\overline{\rho}$. 
We define $U_{\tri} (\overline{\rho})$ as the set of points $(x, \delta) \in {\mathfrak{X}}_{\overline{\rho}} \times {\mathcal{T}}_{\GSp_{2n}, \reg}$ such that for the representation $\rho_x \colon {\mathcal{G}}_K \to \GSp_{2n} (L')$, which is associated with a point $x \in {\mathfrak{X}}_{\overline{\rho}}$, its framed $(\varphi, \Gamma)$-module with $\GSp_{2n}$-structure $D_{\rig} (\rho_x)$ is trianguline with the parameter $\delta \in {\mathcal{T}}_{\GSp_{2n}, \reg}$. 

\begin{definition}
Let $X_{\tri} (\overline{\rho})$ denote the Zariski closure of $U_{\tri} (\overline{\rho})$ in ${\mathfrak{X}}_{\overline{\rho}} \times {\mathcal{T}}_{\GSp_{2n}}$. We add the structure of the reduced closed analytic subvariety (\cite[9.5.3 Proposition 4]{BGR}) of ${\mathfrak{X}}_{\overline{\rho}} \times {\mathcal{T}}_{\GSp_{2n}}$ to $X_{\tri} (\overline{\rho})$. We call the rigid space $X_{\tri} (\overline{\rho})$ the \emph{trianguline deformation space} of $\overline{\rho}$. 
\end{definition}

We write $\omega'$ for the morphism $X_{\tri} (\overline{\rho}) \to {\mathcal{T}}_{\GSp_{2n}}$, which is the composition of the inclusion $X_{\tri} (\overline{\rho}) \hookrightarrow {\mathfrak{X}}_{\overline{\rho}} \times {\mathcal{T}}_{\GSp_{2n}}$ and the second projection. We also write $\omega$ for the composition of $\omega'$ and the restriction ${\mathcal{T}}_{\GSp_{2n}} \to {\mathcal{W}}_{\GSp_{2n}}$. 

\begin{theorem}
\label{Xtri_mainpropoerty}
\begin{enumerate}
\item The rigid space $X_{\tri} (\overline{\rho})$ is equidimensional of dimension $2n^2 + n + 1 + [K \colon {\mathbb{Q}}_p]\dfrac{(n + 1)(n + 2)}{2}$. 
\item The subset $U_{\tri} (\overline{\rho})$ is Zariski open and Zariski dense in $X_{\tri} (\overline{\rho})$. 
\item The open subspace $U_{\tri} (\overline{\rho})$ is smooth over ${\mathbb{Q}}_p$ and the restriction of $\omega'$ to $U_{\tri} (\overline{\rho})$ is smooth morphism. 
\end{enumerate}
\end{theorem}

\begin{proof}
This follows by the same argument of \cite[Th\'eor\`eme 2.6]{BHS15}.
\end{proof}





\section{Triangulations and refinements} 

First, we introduce the notation of crystalline $(\varphi, \Gamma)$-modules over Robba rings (cf.\ \cite[\S 3.3]{HMS18}). 
There exists a left exact functor $D_{\crys}$ from the category of $(\varphi, \Gamma)$-modules over ${\mathcal{R}}_L$ to the category of $K$-filtered $\varphi$-modules over $L$, which is defined by $D_{\crys} (D) \coloneqq (D \left[ \tfrac{1}{t} \right])^{\Gamma}$. We say that a $(\varphi, \Gamma)$-module $D$ over ${\mathcal{R}}_L$ is \emph{crystalline} if ${\dim}_{K_0} D_{\crys} (D) = {\rank}_{{\mathcal{R}}_L} D$. Note that for a crystalline representation $\rho \colon {\mathcal{G}}_K \to \GL (V)$, the module $D_{\rig} (\rho)$ is crystalline and there is an isomorphism 
\[
D_{\crys} (\rho) \coloneqq (B_{\crys} \otimes_{{\mathbb{Q}}_p} V)^{{\mathcal{G}}_K} \cong D_{\crys} (D_{\rig} (\rho))
\]
by \cite[Th\'eor\`em 3.6]{Ber02}. 
Let $D$ be a crystalline $(\varphi, \Gamma)$-module. There exists a bijection between saturated sub-$(\varphi, \Gamma)$-modules $D'$ and $\varphi$-stable sub-$L {\otimes}_{{\mathbb{Q}}_p} K_0$-modules $\mathcal{F}$ of $D_{\crys} (D)$ by the same argument of \cite[\S 2.4.2]{BC09}. The explicit construction of the bijection and its inverse are as follows: 
\[
D' \mapsto D_{\crys} (D') = \left( D' \left[ \tfrac{1}{t} \right] \right)^{\Gamma}, \quad \mathcal{F} \mapsto \left( {\mathcal{R}}_L \left[ \tfrac{1}{t} \right] \cdot \mathcal{F} \right) \cap D. 
\]

A \emph{refinement} of a $K$-filtered $\varphi$-module $D_{\crys} (D)$ of rank $m$ over $L$ is a filtration ${\mathcal{F}}_{\bullet}$ 
\[
{\mathcal{F}}_0 = 0 \subsetneq {\mathcal{F}}_1 \subsetneq \cdots \subsetneq {\mathcal{F}}_m = \mathcal{F}
\]
of $\mathcal{F}$ such that each $F_i$, which is free of rank $i$, is a $L \otimes_{{\mathbb{Q}}_p} K_0$-submodule stable under $\varphi$. Consequently, there is a bijection 
\[
\mathcal{P}  = (D_i)_{0 \leq i \leq m} \mapsto D_{\crys} (\mathcal{P}) \coloneqq (D_{\crys} (D_i))_{0 \leq i \leq m}
\]
between the set of (classical) triangulations of $D$ and the set of refinements of $D_{\crys} (D)$ for crystalline $D$. 

Let $G$ be a split reductive group of regular type over $K$ and fix an algebraic embedding $r \colon G \to \GL_m$ of regular type. We say that a $(\varphi, \Gamma)$-module with $G$-structure $\omega \colon \Rep_K (G) \to \Phi \Gamma ({\mathcal{R}}_L)$ is \emph{crystalline} if $\omega (V)$ is crystalline for some faithful representation $r \colon G \to \GL (V)$, equivalently for any faithful representation $r \colon G \to \GL (V)$ by \cite[Theorem C.1.7]{Lev13}. 
Let $D$ be a crystalline $(\varphi, \Gamma)$-module with $G$-structure $\Rep_K (G) \to \Phi \Gamma ({\mathcal{R}}_L)$.  Since $D_{\crys}$ is exact on its restriction to the category of crystalline $(\varphi, \Gamma)$-modules, the composition 
\[
\Rep_K (G) \to \Phi \Gamma ({\mathcal{R}}_L) \xrightarrow{D_{\crys}} \Phi F_K (L)
\]
is also an exact tensor functor. Write $D_{\crys} (D)$ for the composition functor. Though the underlying $G$-torsor of the $D_{\crys} (D)$ is not in general a trivial torsor over $K_0 \otimes_{{\mathbb{Q}}_p} L$, it becomes trivial after replacing $K$ and $L$ by a sufficiently large finite extension. A \emph{refinement} $\mathfrak{P}$ of $D_{\crys} (D)$ is a sub-$B$-torsor of $D_{\crys} (D)$ which is invariant under the Frobenius action $\Phi$. 



\begin{proposition}
\label{G_Berger_corr}
Let $D$ be a crystalline frameable $(\varphi, \Gamma)$-module with $G$-structure over ${\mathcal{R}}_L$. Assume the Frobenius action ${\Phi}^{[K_0 \colon {\mathbb{Q}}_p]}$ on $r_{\ast} (D_{\crys} (D))$ has $m$ distinct eigenvalues in $L$. Then there is a bijection from the set of triangulations of $D$ to that of refinements of $D_{\crys} (D)$. Moreover, if there is a refinement of $D_{\crys} (D)$, there exists exactly $\# W(G, T)$ refinements of $D_{\crys} (D)$. 
\end{proposition}

\begin{proof}
Recall that a triangulation $\mathcal{P}$ of $D$ can be seen as an exact tensor functor $\omega_{\mathcal{P}} \colon \Rep_K (B) \to \Phi \Gamma ({\mathcal{R}}_L)$ whose composition with the restriction $\Rep_K (G) \to \Rep_K (B)$ is identical to $D$. We construct a map from the set of triangulations of $D$ to that of refinements of $D_{\crys} (D)$ as $\omega_{\mathcal{P}} \mapsto \omega_{\mathcal{P}} \circ D_{\crys}$. 

We construct an inverse. By Lemma \ref{parab_ext_class}, we may assume that the underlying $G$-torsor of $D_{\crys} (D)$ is trivial. A refinement of $D_{\crys} (D)$ gives the (classical) refinement of $r_{\ast} (D_{\crys} (D))$ and it induces the (classical) triangulation of $r_{\ast} (D)$ by the classical correspondence between triangulations and refinements. It suffices to show that the triangulation is isomorphic to the form $r_{\ast} (\mathcal{P})$ for some unique triangulation $\mathcal{P}$ of $D$. 

For a commutative $K$-algebra $A$, there is the map $G(A)/B(A) \to (G/B)(A)$ induced from the exact sequence of pointed sets 
\[
1 \to B(A) \to G(A) \to (G/B)(A) \to H^1 (A, B)
\]
by \cite[\S 1.2.2]{Ces22}, where $G/B$ is the flag variety of $(G, B, T)$. Assume that $A$ is a B\'ezout domain. Since the Picard group of $A$ is trivial, it follows $H^1 (A, B) \cong H^1 (A, T) \cong 1$ and the map $G(A)/B(A) \to (G/B)(A)$ is bijective by Lemma \ref{parab_ext_class}. In particular, we have $G({\mathcal{R}}_L)/B({\mathcal{R}}_L) = (G/B)({\mathcal{R}}_L)$ and $G({\mathcal{R}}_L \left[ \tfrac{1}{t} \right])/B({\mathcal{R}}_L \left[ \tfrac{1}{t} \right]) = (G/B)({\mathcal{R}}_L \left[ \tfrac{1}{t} \right])$. The faithful representation $r$ induces the proper morphism $G/B \to \GL_m / B_m$ of the flag varieties over $K$. It is a monomorphism of fppf-sheaves by the definition of $r$ and also a monomorphism of schemes. So it is a closed immersion by \cite[Tag 04XV]{SP}. The variety $G/B$ is smooth and then reduced as well. Then, there exists a unique dotted arrow such that the diagram 
\[\xymatrix{
\Spec \left( {\mathcal{R}}_L \left[ \tfrac{1}{t} \right] \right) \ar[r] \ar[d] & G/B \ar[d] \\ 
\Spec ({\mathcal{R}}_L) \ar@{-->}[ru] \ar[r] & \GL_m / B_m, 
}\]
is commutative by the argument of the scheme theoretic images. Thus, the first part of the proposition is shown. 

For the last part, we may see that if there exist a refinement $xB$ for some $x$ in the frameable $\varphi$-module with $B$-structure as a $\varphi$-invariant sub-$B$-torsor of the frameable $\varphi$-module with $G$-structure associated with $D_{\crys} (D)$, the set of refinements of $D_{\crys} (D)$ is exactly $\{ xwB \ | \ w \in W(G, T) \}$ by the definition of the embedding $G \to \GL_m$ of regular type and the Bruhat decomposition. 
\end{proof}

\begin{remark}
In the case of $G = \GSp_{2n}$, after replacing $K$ and $L$ with sufficient large finite extensions, there exists a refinement of $D_{\crys} (D)$ under the condition of Proposition \ref{G_Berger_corr}. 
\end{remark}


Let $D$ be a $(\varphi, \Gamma)$-module over ${\mathcal{R}}_L$ of rank $m$. 
If $D$ is crystalline, we write 
\begin{align*}
\mathcal{F} &\coloneqq D_{\crys} (D), \\ 
{\mathcal{F}}_K &\coloneqq D_{\crys} (D)_K \coloneqq D_{\crys} (D) \otimes_{K_0} K, \\ 
{\mathcal{F}}_{K, \tau} &\coloneqq {\mathcal{F}}_K \otimes_{K \otimes_{{\mathbb{Q}}_p} L, \tau} L, \\ 
\Fil^{\bullet} &\coloneqq (\Fil^i)_{i \in \mathbb{Z}} = (\Fil^i {\mathcal{F}}_K)_{i \in \mathbb{Z}}
\end{align*}
for the Hodge filtration on ${\mathcal{F}}_K$ and $\gr^i \coloneqq \Fil^i / \Fil^{i+1}$. 
For a refinement 
\[
{\mathcal{F}}_{\bullet} \coloneqq ({\mathcal{F}}_0 = 0 \subsetneq {\mathcal{F}}_1 \subsetneq \cdots \subsetneq {\mathcal{F}}_m = \mathcal{F})
\]
of $\mathcal{F}$, 
we associate it with an ordering 
\[
(\varphi_{\tau, 1}, \ldots, \varphi_{\tau, m})_{\tau} \in (L^{\oplus m})^{\Sigma} \cong (K_0 \otimes_{{\mathbb{Q}}_p} L)^{\oplus m}
\]
of the eigenvalues of $\Phi^{[K_0 \colon {\mathbb{Q}}_p]}$ on $\mathcal{F}$ such that the formula 
\[
\det (T - {\Phi^{[K_0 \colon {\mathbb{Q}}_p]}} |_{{\mathcal{F}}_i}) = \prod_{j = 1}^i (T - \varphi_j), 
\]
holds and 
\[
\mathbf{k} \coloneqq ({\mathbf{k}}_{\tau, 1}, \ldots, {\mathbf{k}}_{\tau, m})_{\tau} \in ({\mathbb{Z}}^{\oplus m})^{\Sigma}
\]
such that the quotient ${\mathcal{F}}_{K, \tau, i} / {\mathcal{F}}_{K, \tau, i-1}$, where ${\mathcal{F}}_{K, \tau, i} \coloneqq {\mathcal{F}}_i \otimes_{K_0 \otimes_{{\mathbb{Q}}_p} L, \ \tau} L$, of the filtered modules over $L$ has the non-zero graded quotient in degree $-{\mathbf{k}}_{\tau, i}$ for any $\tau \in \Sigma$ and $i \in [1, m]$. 
If $\mathbf{k} = (k_{\tau})_{\tau \colon K \hookrightarrow L} \in {\mathbb{Z}}^{\Sigma}$, let $z^{\mathbf{k}} \in \mathcal{T} (L)$ denote the character 
\[
z \mapsto \prod_{\tau \in \Hom (K, L)} \tau (z)^{k_{\tau}}
\]
where $z \in K^{\ast}$. For 
\[
\mathbf{k} = (k_{\tau, i})_{i \in [1, m], \tau \colon K \hookrightarrow L} \in ({\mathbb{Z}}^{\oplus m})^{\Sigma}, 
\]
let $z^{\mathbf{k}} \in {\mathcal{T}}^m (L)$ also denote the character 
\[
(z_1, \ldots, z_m) \mapsto \prod_{\substack{i \in [1, m] \\ \tau \in \Hom (K, L)}} \tau (z_i)^{k_{\tau, i}}
\]
where $(z_1, \ldots, z_m) \in (K^{\ast})^m$. 
We write $\unr (a)$ for the character $K^{\ast} \to A^{\ast}$ which sends ${\mathcal{O}}_K^{\ast}$ to $1$ and $\varpi_K$ to $a$ for an $L$-algebra $A$ and $a \in A^{\ast}$. 
By the argument of \cite[Proposition 2.4,1]{BC09} and \cite[Lemma 2.1]{BHS15}, we have the following lemma. 

\begin{lemma}
\label{Berger_parameter}
For a $(\varphi, \Gamma)$-module $D$ over ${\mathcal{R}}_L$ of rank $m$ and a refinement of $D_{\crys} (D)$, the corresponding triangulation of $D$ has the parameter $(\delta_1, \ldots, \delta_m)$ such that $\delta_i = z^{{\mathbf{k}}_i} \unr (\varphi_i)$ where $(\varphi_i)_i$ and ${\mathbf{k}}_i$ are constructed above. 
\end{lemma}

We introduce a notion of $G$-valued strictly dominance and noncriticality of triangulations of crystalline $(\varphi, \Gamma)$ modules with $G$-structures. 
We identify ${\mathcal{W}}_G$ (resp.\ ${\mathcal{T}}_G$) with ${\mathcal{W}}^d$ (resp.\ ${\mathcal{T}}^d$) by the fixed isomorphism $G \supset T \cong {\mathbb{G}}_m^d$ of algebraic tori. 
Let $D$ be a crystalline frameable $(\varphi, \Gamma)$-module with $G$-structure and $\widetilde{D} \coloneqq r_{\ast} (D)$ be a faithful representation $r \colon G \to \GL_m$ of regular type. We assume that the underlying $G$-torsor of $D_{\crys} (\widetilde{D})$ is trivial and the rank $\gr^i ((D_{\crys} (\widetilde{D}) \otimes_{K_0} K) \otimes_{K \otimes_{{\mathbb{Q}}_p} L, \tau} L)$ is at most $1$ for any $\tau \in \Sigma$.  For a rank one $K$-filtered $\varphi$-module $\mathcal{F}$ over $L$ and $\tau \in \Sigma$, let $\deg_{\tau} (\mathcal{F})$ denote the integer $i$ such that $\gr^j (\mathcal{F} \otimes_{K \otimes_{{\mathbb{Q}}_p} L, \tau} L)$ is non-trivial. Note that the assumption is satisfied if $D$ is trianguline. Here we forget the $\Phi$-structure of $\mathcal{F}$ and so a refinement of $\mathcal{F}$ means a structure of sub-$B$-torsor in $\mathcal{F}$, which is equivalent to a flag in $(G/B)(K \otimes_{{\mathbb{Q}}_p} L)$ if we fix a basis of $\widetilde{\mathcal{F}}$. 

For a refinement 
\[
{\widetilde{\mathcal{F}}}_{\bullet} \coloneqq ({\widetilde{\mathcal{F}}}_0 = 0 \subsetneq {\widetilde{\mathcal{F}}}_1 \subsetneq \cdots \subsetneq {\widetilde{\mathcal{F}}}_m = D_{\crys} (\widetilde{D}))
\]
of $D_{\crys} (\widetilde{D})$, we put 
\[
(d_{\tau, 1}, \ldots, d_{\tau, m}) \coloneqq (\deg_{\tau} ({\widetilde{\mathcal{F}}}_1), \deg_{\tau} ({\widetilde{\mathcal{F}}}_2 / {\widetilde{\mathcal{F}}}_1), \ldots, \deg_{\tau} ({\widetilde{\mathcal{F}}}_m / {\widetilde{\mathcal{F}}}_{m-1})) 
\]
and $(d_{\dom, \tau, 1}, \ldots, d_{\dom, \tau, m})$ to be the replacement of the factors of $(d_{\tau, 1}, \ldots, d_{\tau, m})$ such that $d_{\dom, \tau, 1} \leq \cdots \leq d_{\dom, \tau, m}$. We set Zariski open loci 
\[
\GL_m / B_m \eqqcolon U_{\tau}^0 \supseteq U_{\tau}^1 \supseteq \cdots \supseteq U_{\tau}^{m-1} = U_{\tau}^m \supseteq U_{\tau}^{m+1} \coloneqq \emptyset 
\]
of the flag variety as a scheme $\GL_m / B_m$ over $L$ such as $U_{\tau}^l$ for $l \in [1, m]$ is the Zariski open subset of the points such that $d_{\tau, i} = d_{\dom, \tau, i}$ for $i \in [1, l]$ in $\GL_m / B_m$ and $U_{\tau}^{m+1} \coloneqq \emptyset$. 
Moreover, we take the stratification 
\[
U_{\tau}^{i-1} = \bigsqcup_{j = i}^{m} U_{\tau}^{i, j} 
\]
where $U_{\tau}^{i, j}$ is the subset of points such that $d_{\tau, i} = d_{\dom, \tau, j}$ in $U_{\tau}^i$ for $i \in [1, m]$. Note that $U_{\tau}^{i, k}$ is Zariski open in $U_{\tau}^{i-1} \backslash \bigsqcup_{j = i}^{k-1} U_{\tau}^{i, j}$. 
We write the closed immersion $\overline{r} \colon G/B \to \GL_m / B_m$ for the map of flag varieties induced from $r \colon G \to \GL_m$. Then, we construct $(d'_{\dom, \tau, 1}, \ldots, d'_{\dom, \tau, m}) \in {\mathbb{Z}}^m$ from the data 
\[
(m, (d_{\tau, 1}, \ldots, d_{\tau, m}), (d_{\dom, \tau, 1}, \ldots, d_{\dom, \tau, m}), (U_{\tau}^{i, j})_{1 \leq i \leq j \leq m+1}, \overline{r})
\]
where $m \in {\mathbb{Z}}_{\geq 1}$, $(d_{\tau, 1}, \ldots, d_{\tau, m}), (d_{\dom, \tau, 1}, \ldots, d_{\dom, \tau, m}) \in {\mathbb{Z}}^m$ and $U_{\tau}^{i, j} \subset \GL_m / B_m$ with the following algorithm: 

\begin{enumerate}

\item[Step 1.] Find $k \in [0, m]$ such that ${\overline{r}}^{-1} (U_{\tau}^k) \cap (G/B) \neq \emptyset$ and ${\overline{r}}^{-1} (U_{\tau}^{k+1}) \cap (G/B) = \emptyset$. If $k = m$, the algorithm finishes. 

\item[Step 2.] Find $j \in [k, m]$ such that ${\overline{r}}^{-1} (U_{\tau}^{k, j}) \cap (G/B) \neq \emptyset$ and ${\overline{r}}^{-1} (U_{\tau}^{k, j+1}) \cap (G/B) = \emptyset$. Then we set $d'_{\dom, \tau, i} \coloneqq d_{\dom, \tau, i}$ for $i \in [1, k]$ and $d'_{\dom, \tau, k+1} \coloneqq d_{\dom, \tau, j}$. If $k = m - 1$, the algorithm finishes. 

\item[Step 3.] Replacing the data 
\[
(m, (d_{\tau, 1}, \ldots, d_{\tau, m}), (d_{\dom, \tau, 1}, \ldots, d_{\dom, \tau, m}), (U_{\tau}^{i, j})_{1 \leq i \leq j \leq m+1}, \overline{r})
\]
with the data 
\[
(m - k - 1, (d_{\tau, k+1}, \ldots, d_{\tau, j-1}, d_{\tau, j+1}, \ldots, d_{\tau, m}), (d_{\dom, \tau, k+2}, \ldots, d_{\dom, \tau, m}), (U_{\tau}^{\prime i, j})_{ij}, \overline{r})
\]
where $(U_{\tau}^{\prime i, j})_{ij} = (U_{\tau}^{\prime i, j})_{1 \leq i \leq j \leq m - k - 1}$ is constructed from 
\[
(d_{\tau, k+1}, \ldots, d_{\tau, j-1}, d_{\tau, j+1}, \ldots, d_{\tau, m}), (d_{\dom, \tau, k+2}, \ldots, d_{\dom, \tau, m})
\]
in the manner described above. 

\item[Step 4.] Return to the Step 1. 

\end{enumerate}

The vectors of weights $(d'_{\dom, \tau, 1}, \ldots, d'_{\dom, \tau, m})$ for $\tau \in \Sigma$ define the point $(z^{d'_1}, \ldots, z^{d'_m}) \in {\mathcal{T}}^m (L) = {\mathcal{T}}_{\GL_m} (L)$ where $d'_i \coloneqq (d'_{\dom, \tau, i})_{\tau \in \Sigma}$. There exists a unique $d'_{G, i} \coloneqq (d'_{G, \dom, \tau, i})_{\tau \in \Sigma} \in {\mathbb{Z}}^{\Sigma}$ for each $i \in [1, d]$, where $d \coloneqq \dim (T)$, such that the pushout $r_{T, {\ast}} (D)$ of the frameable $(\varphi, \Gamma)$-module with $T$-structure $D_T$ corresponding to $(z^{d'_{G, 1}}, \ldots, z^{d'_{G, d}}) \in ({\mathcal{W}}_G (L))^{\Sigma}$ by the inclusion of tori $r_T \colon T \to T_m$ has the parameter $(z^{d'_1}, \ldots, z^{d'_m}) \in {{\mathcal{W}}^m (L)}^{\Sigma}$. 
By the construction, the locus of the points of refinements in $G/B$ which has the parameter $(z^{d'_{G, 1}}, \ldots, z^{d'_{G, d}})$ is Zariski open in the topology of schemes. If $d_{\tau, i}  \neq d_{\tau, j}$ for each $i \neq j$ and $\tau \in \Sigma$, we call the weight $(d'_{G, 1}, \ldots, d'_{G, d}) ({\mathbb{Z}}^d)^{\Sigma}$ or the parameter $(z^{d'_{G, 1}}, \ldots, z^{d'_{G, d}}) \in ({\mathcal{W}}_G)^{\Sigma}$ (resp.\ a parameter in $({\mathcal{W}}_G)^{\Sigma}$ whose image in $({\mathcal{T}}_G)^{\Sigma}$ is strictly dominant) \emph{strictly dominant}. 

From here in this section, let $G = \GSp_{2n}$ and $\overline{\rho} \colon \mathcal{G}_K \to G(k_L) = \GSp_{2n} (k_L)$ be a $\GSp_{2n} (k_L)$-valued representation. 
We say that $x = (\rho, \underline{\delta} = (\delta_1, \ldots, \delta_m)) \in X_{\tri} (\overline{\rho})$ is \emph{crystalline} if $\rho \colon {\mathcal{G}}_K \to \GSp_{2n} (k(x))$ is a $\GSp_{2n}$-valued crystalline representation. 

We also call a point $x = (\rho, \delta) \in X_{\tri} (\overline{\rho})$ \emph{strictly dominant} if $\omega (x)$ is strictly dominant. 
For a $(\varphi, \Gamma)$-module with $G$-structure $D$ such that $r_{\ast} D$ has distinct $m$ Hodge-Tate weights and its triangulation $\mathcal{P}$,
we say that a triangulation $\mathcal{P}$ of a crystalline $(\varphi, \Gamma)$-module with $G$-structure $D$ with a regular Hodge-Tate weight such that $D_{\crys} (r_{\ast} (D))$ has $m$ distinct Frobenius eigenvalues is \emph{noncritical} if the corresponding parameter $(z^{d'_{G, 1}}, \ldots, z^{d'_{G, d}}) \in {\mathcal{W}}_G$ constructed as above from the refinement of $D_{\crys} (D)$ associated with $\mathcal{P}$ is strictly dominant. 

\begin{example}
In the case of $G = \GSp_{2n}$ and $T$ is the maximal split torus of $\GSp_{2n}$, the inclusion of tori $T \to T_{2n}$ induced by the standard representation $\GSp_{2n} \to \GL_{2n}$ determines that strictly dominant parameters are the form $(z^{d'_{G, 1}}, \ldots, z^{d'_{G, n+1}}) \in {\mathcal{W}}_{\GSp_{2n}}^{\Sigma}$ such that $d'_{G, \tau, 1} > \cdots > d'_{G, \tau, n}$ and $2d'_{G, \tau, n} > d'_{G, \tau, n+1}$ for all $\tau \in \Sigma$. 
\end{example}

We call a point $x = (\rho, \delta_1, \ldots, \delta_{n+1}) \in X_{\tri} (\overline{\rho})$ \emph{saturated} if there exists a triangulation of the $(\varphi, \Gamma)$-module with $\GSp_{2n}$-structure $D_{\rig} (\rho)$ with parameter $(\delta_1, \ldots, \delta_{n+1})$. 

\begin{lemma}
\label{ext_saturated}
Let $\overline{\rho} \colon {\mathcal{G}}_K \to \GSp_{2n} (k_L)$ be a $\GSp_{2n} (k_L)$-valued representation and $x = (r, \delta_1, \ldots, \delta_{n+1}) \in X_{\tri} (\overline{\rho})$ be a strictly dominant point with $\omega (x) = \delta_{\mathbf{k}}$ for some $\mathbf{k} = (k_{\tau, i})_{i \in [1, n], \tau \in \Sigma} \in ({\mathbb{Z}}^n)^{\Sigma}$. Assume that: 
\[
k_{\tau, i} - k_{\tau, i+1} > [K \colon K_0] \val(\delta_1 (\varpi_K) \cdots \delta_i (\varpi_K)), 
\]
where $\val(-)$ is the function of $p$-adic values. for $i \in [1, n]$, $\tau \in \Sigma$. Then $x$ is saturated and semi-stable. Moreover, if $(\delta_1, \ldots, \delta_{n+1}) \in {\mathcal{T}}_{G, \reg}$, then $r$ is crystalline strictly dominant noncritical. 
\end{lemma}

\begin{proof}
This follow by the above construction and the proof of \cite[Lemma 2.10]{BHS15}. 
\end{proof}

\section{Explicit calculation of the tangent spaces of deformation functors}
\label{explicit_calculation}

We prove the additive version of the reconstruction theorem (\cite[Theorem 11.2]{MilAG}). 
Let $G$ be an affine algebraic monoid over a field $k$ and $R$ be a $k$-algebra. For a free module $V_R$ over $R$, each element $Y \in \Lie_R (\GL (V_R))$ induces an endomorphism of the $R$-module $V_R$ by the natural isomorphism $\Lie_R (\GL (V_R)) \xrightarrow{\sim} \End_R (V_R)$. We also write $Y \colon V_R \to V_R$ for the induced endomorphism by $Y \in \Lie_R (\GL (V_R))$. 
Let $X \in (\Lie_R (G))(R)$. For every finite-dimensional algebraic representation $r_V \colon G \to \GL (V)$ of $G$ over $k$, we have an $R$-linear map $\lambda_V \coloneqq (\Lie_R (r_{V, R}))(X) \colon V_R \to V_R$. These maps satisfy the following conditions: 

\begin{enumerate}
\item[(C1)] For all finite-dimensional representations $V$, $W$, 
\[
\lambda_{V \otimes W} = \lambda_V \otimes 1 + 1 \otimes \lambda_W. 
\]
\item[(C2)] If we write $\mathbf{1}$ for the $1$-dimensional trivial representation of $G$, induced $\lambda_{\mathbf{1}} \colon k \to k$ is the zero map. 
\item[(C3)] For all $G$-equivariant maps $u \colon V \to W$, 
\[
\lambda_W \circ u_R = u_R \circ \lambda_V. 
\]
\end{enumerate}

We write $G_R \coloneqq G \times_{\Spec (k)} \Spec (R)$ and $V_R \coloneqq V \otimes_k R$ from here. 

\begin{lemma}
\label{additive_reconst}
Let $G$ be an affine algebraic monoid over a field $k$ and $R$ be a $k$-algebra. Suppose that we are given an $R$-linear map $\lambda_V \colon V_R \to V_R$ for every finite-dimensional algebraic representation $r_V \colon G \to \GL (V)$ of $G$. If the family $(\lambda_V)_V$ satisfies the conditions (C1), (C2), (C3) above, there exists a unique $X \in \Lie_R (G_R)$ such that $\lambda_V = ((\Lie_R (r_{V, R}))(X) \colon V_R \to V_R)$ for all $V$. 
\end{lemma}

\begin{proof}
The conditions (C1), (C2), (C3) are satisfied for the uniquely induced $\lambda_V \colon V_R \to V_R$ from the enlarged family of filtered unions $V$ of the finite-dimensional representations. 
We set $A \coloneqq \Gamma (G, {\mathcal{O}}_G)$. Considering $A$ as a $k$-vector space, the regular representation $r_A \colon G \to \GL (A)$ is a filtered union of its finite-dimensional subrepresentation by \cite[Proposition 4.6]{MilAG}. We write $\lambda_A \colon A \to A$ for the $k$-linear map induced from $r_A$. 

By the condition (C1) and (C3), the diagram 
\[\xymatrix{
A \otimes_k A \ar[r]^-{m} \ar[d]_{\lambda_A \otimes 1 + 1 \otimes \lambda_A} & A \ar[d]^{\lambda_A} \\ 
A \otimes_k A \ar[r]_-{m} & A, 
}\]
where $m$ is the multiplication structure of the monoid $G$, is commutative, so that $r_A$ is an element of the module $\Der_k (A, A)$ of $k$-derivations. We also have the natural isomorphism 
\[
\iota \colon \Lie_k (G) \xrightarrow{\sim} \Der_k (A, A)
\]
of $k$-vector spaces by \cite[Proposition 12.23, Proposition 12.24]{MilAG}. 

For an element $X \in \Lie_R (G_R)$ and a finite-dimensional algebraic representation $\lambda_V \colon G \to \GL (V)$, $X$ here gives an element $X_{\epsilon} \in \GL (V[\epsilon])$. Writing $\partial X$ for the composition 
\[
V \hookrightarrow V \oplus \epsilon V \xrightarrow{X_{\epsilon}} V \oplus \epsilon V \twoheadrightarrow \epsilon V \xrightarrow{\sim} V, 
\]
The element $(\Lie_R (r_{V, R})) \in \Lie (\GL (V_R))$ is sent to $\partial X \in \End_R (V_R)$ by $\Lie_R (\GL (V_R)) \xrightarrow{\sim} \End_R (V_R)$. We may see that the composition 
\[
\Lie_R (G_R) \to \Lie_R (\GL (A_R)) \xrightarrow{\sim} \End_R (V_R)
\]
is the same as the composition of $\iota_R$ and the inclusion $\Der_R (A_R, A_R)$. Then there exists $X \in \Lie_R (G_R)$ such that $\lambda_A = (\Lie_R (r_{A, R}))(X)$. 

For each finite-dimensional algebraic representation $r_V \colon G \to \GL (V)$, taking $V_0$ for the underlying vector space of $V$ that has the trivial $G$-action, there exists an injection $i \colon V \hookrightarrow V_0 \otimes A$ of $k$-vector spaces with $G$-actions by \cite[Proposition 4.9]{MilAG}. In the diagram 
\[\xymatrix{
V \ar[r]^-{i} \ar[d]_{\lambda_V} & V_0 \otimes A \ar[d]^{\lambda_{V_0 \otimes A}} \\ 
V \ar[r]_-{i} & V_0 \otimes A, 
}\]
$\lambda_V$ is also the endomorphism corresponding to the element $(\Lie_R (r_{V, R}))(X) \in \Lie_R (\GL (V_R))$ by the conditions (C1), (C2), (C3). Then it follows that the existence of $X \in \Lie_R (G_R)$. The uniqueness of $X$ follows from considering $\lambda_{W_0 \otimes A}$ for a finite-dimensional faithful representation $r_W \colon G \to \GL (W)$ and its underlying $k$-vector space $W_0$. 
\end{proof}

Let $G$ be an affine algebraic group over a field $k$ and $X$ be a $k$-scheme. We recall the notation of exact tensor filtrations of the exact tensor functors $\omega \colon \Rep_K (G) \to \Bun (X)$ (cf.\ \cite{SR72}, \cite[\S 2.7]{BG17}). 
An \emph{exact tensor filtration} of $\omega$ is the data of decreasing filtrations $({\Fil}^{\bullet} (\omega (V)))_V$ for each $V \in \Rep_K (G)$ such that the following condition are satisfied: 

\begin{enumerate}
\item the filtrations are functorial in $V$. 
\item the filtrations are tensor compatible, in the sense that 
\[
{\Fil}^n \omega (V \otimes_k V') = \sum_{p+q=n} ({\Fil}^p \omega (V) \otimes_{{\mathcal{O}}_X} {\Fil}^q \omega (V')) \subset \omega (V) \otimes_{{\mathcal{O}}_X} \omega (V'). 
\]
\item For the $1$-dimensional trivial representation $\mathbf{1} \in \Rep_K (G)$ it suffices that ${\Fil}^n (\omega (\mathbf{1})) = \omega (\mathbf{1})$ if $n \leq 0$ and ${\Fil}^n (\omega (\mathbf{1})) = 0$ if $n \geq 1$. 
\item The associated functor from $\Rep_K (G)$ to the category of $\mathbb{Z}$-graded objects of $\Bun (X)$ is exact. 
\end{enumerate} 

For an exact tensor functor $\omega \colon \Rep_K (G) \to \Bun (X)$ and a morphism $Y \to X$ of schemes, let $\omega_Y \colon \Rep_K (G) \to \Bun (Y)$ be the composition of $\omega$ and the base change functor $\Bun (X) \to \Bun (Y)$. If $Y = \Spec (R)$ for a commutative ring $R$, we write $\omega_R = \omega_Y$. 

\begin{definition}
Let $\omega, \eta \colon \Rep_K (G) \rightrightarrows \Bun (X)$ be exact tensor functors. Then ${\underline{\Hom}}^{\otimes} (\omega, \eta)$ is the presheaf on affine schemes $\Spec (R)$ over $X$ given by 
\[
{\underline{\Hom}}^{\otimes} (\Spec (R)) \coloneqq \Hom^{\otimes} (\omega_R, \eta_R), 
\]
where ${\underline{\Hom}}^{\otimes}$ is the set of natural transformations of functors which preserves tensor products. 
\end{definition}

We define two subfunctors of ${\underline{\Aut}}^{\otimes} (\omega)$ dependent on a fixed exact tensor filtration ${\Fil}^{\bullet}$ of $\omega$. 

\begin{definition}
Let ${\Fil}^{\bullet}$ be a fixed exact tensor filtration of $\omega$. 
\begin{enumerate}
\item Let $P_{{\Fil}^{\bullet}} = {\underline{\Aut}}_{{\Fil}^{\bullet}}^{\otimes} (\omega)$ be the functor on $R$-algebras such that 
\[
R' \mapsto \{ \lambda \in {\underline{\Aut}}^{\otimes} (\omega) (R') \ | \ \lambda ({\Fil}^n \omega (V)) \subset {\Fil}^n \omega (V) \ \text{for all} \ V \in \Rep_K (G) \ \text{and} \ n \in \mathbb{Z} \}. 
\]
\item Let $U_{{\Fil}^{\bullet}} = {\underline{\Aut}}_{{\Fil}^{\bullet}}^{\otimes !} (\omega)$ be the functor on $R$-algebras such that 
\[
R' \mapsto \{ \lambda \in {\underline{\Aut}}^{\otimes} (\omega) (R') \ | \ (\lambda - \text{id}) ({\Fil}^n \omega (V)) \subset {\Fil}^{n+1} \omega (V) \ \text{for all} \ V \in \Rep_K (G) \ \text{and} \ n \in \mathbb{Z} \}. 
\]
\end{enumerate}
\end{definition}

Let $\widetilde{X} \to X$ be $G$-torsor corresponding to $\omega$. 

\begin{proposition}
\label{tensor_auto_represent}
\begin{enumerate}
\item The functor ${\underline{\Aut}}^{\otimes} (\omega)$ is representable by the $X$-group scheme $\Aut_G (\widetilde{X})$, which is a form of $G \times_k X$. 
\item The functors $P_{{\Fil}^{\bullet}}$ and $U_{{\Fil}^{\bullet}}$ are representable by closed subschemes of $\Aut_G (\widetilde{X})$ and they are smooth if $G$ is. 
\item If $X = \Spec (R)$, we have $\Lie_R P_{{\Fil}^{\bullet}} = {{\Fil}^{\bullet}}^0 (\Lie_R {\underline{\Aut}}^{\otimes} (\omega))$ and $\Lie_R U_{{\Fil}^{\bullet}} = {{\Fil}^{\bullet}}^1 (\Lie_R {\underline{\Aut}}^{\otimes} (\omega))$. 
\end{enumerate}
\end{proposition}

\begin{proof}
See \cite[Chapter IV, 2.1.4.1]{SR72}, \cite[Corollary 2.7.3]{BG17} and its following argument. 
\end{proof}

We assume and $X$ is an affine scheme. Let $\ad (\widetilde{X}) \coloneqq \Lie_R {\underline{\Aut}}^{\otimes} (\omega)$. For a parabolic subgroup $P$ of $\Aut_G (\widetilde{X})$, let $\ad_P (\widetilde{X}) \coloneqq {\Fil}^0 (\Lie_R {\underline{\Aut}}^{\otimes} (\omega))$ for an exact tensor filtration ${\Fil}^{\bullet}$ such that $P_{{\Fil}^{\bullet}} = P$. 
For a $(\varphi, \Gamma)$-module with $G$-structure $D$ over ${\mathcal{R}}_L$ and its triangulation $\mathcal{P}$, we define $\ad_{\mathcal{P}} (D) \coloneqq \ad_P (D)$ where $P$ is the subgroup of $\Aut_G (D)$ which preserves $\mathcal{P} \subset D$. If $D$ is frameable, $\ad_{\mathcal{P}} (D) \subset \Aut_G (D) \cong G_{{\mathcal{R}}_L}$ is a Borel subgroup of $G_{{\mathcal{R}}_L}$ and is a conjugate of the standard Borel subgroup $B_{{\mathcal{R}}_L} \subset G_{{\mathcal{R}}_L}$. 

We define the trianguline deformation functors of a frameable $(\varphi, \Gamma)$-module with $G$-structure $D$ (associated to a triangulation $\mathcal{P}$ of $D$) 
\[
X_D, X_{D, \mathcal{P}} \colon {\mathcal{C}}_L \to \Set
\]
as follows. For an object $A \in {\mathcal{C}}_L$, let $X_{D, \mathcal{P}} (A)$ be the set of isomorphism classes of triples $(D_A, {\mathcal{P}}_A, \pi)$ where $D_A$ is a $(\varphi, \Gamma)$-module with $G$-structure over ${\mathcal{R}}_A$, ${\mathcal{P}}_A$ is a triangulation of $D_A$ and $\pi \colon D_A \otimes_A L \xrightarrow{\sim} D$ is an isomorphism of $(\varphi, \Gamma)$-module with $G$-structure over ${\mathcal{R}}_L$ such that $\pi ({\mathcal{P}}_A) = \mathcal{P}$. Let $X_D (A)$ be the set of isomorphism classes of $(D_A, \pi)$. 
Note that the underlying torsor of a deformation $D_A$ is trivial over ${\mathcal{R}}_A$ as long as $D$ is frameable by \cite[Proposition 6.1.1]{Ces22}. 

\begin{lemma}
\label{tangent_H1adD}
There is a natural isomorphism 
\begin{enumerate}
\item $TX_D = X_D (L[\epsilon]) \xrightarrow{\sim} H^1 (\ad (D))$. 
\item $TX_{D, \mathcal{P}} = X_{D, \mathcal{P}} (L[\epsilon]) \xrightarrow{\sim} H^1 (\ad_{\mathcal{P}} (D))$. 
\end{enumerate}
\end{lemma}

\begin{proof}
We fix a framing $([\phi], [\gamma] \ | \ \gamma \in \Gamma)$ of $D$. For each algebraic representation $r \colon G \to \GL (V)$, the pushout $r_{\ast} (D)$ is then the framed $(\varphi, \Gamma)$-module $([\phi_V], [\gamma_V] \ | \ \gamma \in \Gamma)$ of rank $\dim V$. For algebraic representations $r \colon G \to \GL (V)$, $r' \colon G \to \GL (V')$, lift of those of coefficient $L[\epsilon]$ is the form 
\begin{align*}
&([\phi_V + \epsilon \cdot \partial \phi_V], [\gamma_V + \epsilon \cdot \partial \gamma_V] \ | \ \gamma \in \Gamma), \\ 
&([\phi_{V'} + \epsilon \cdot \partial \phi_{V'}], [\gamma_{V'} + \epsilon \cdot \partial \gamma_{V'}] \ | \ \gamma \in \Gamma), 
\end{align*}
where $\partial \phi_V, \partial \gamma_V \in \End (V)$ and $\partial \phi_{V'}, \partial \gamma_{V'} \in \End (V')$. 

Then for the pushout $(r \otimes r')_{\ast} (D)$ of $D$ by $r \otimes r' \colon G \to \GL (V \otimes V')$, $\phi_{V \otimes V'}$ and $\gamma_{V \otimes V'}$ is written as follows: 
\begin{align*}
\phi_{V \otimes V'} &= \phi_V \otimes \phi_{V'} + \epsilon \cdot (\partial \phi_V \otimes 1 + 1 \otimes \partial \phi_{V'}), \\ 
\gamma_{V \otimes V'} &= \gamma_V \otimes \gamma_{V'} + \epsilon \cdot (\partial \gamma_V \otimes 1 + 1 \otimes \partial \gamma_{V'}).
\end{align*}
By Lemma \ref{additive_reconst}, functorial parameters $\partial \phi_V$, $\partial \gamma_V$ for all $V \in \Rep_K (G)$ is just given by an element of $\Lie_{{\mathcal{R}}_L} (D)$. The results follows by the proof of \cite[Proposition 3.6, (ii)]{Che11}. 
\end{proof}

From here we consider the case of $G = \GL_m$ over $k = K$. For $\mathbf{k} = (k_{\tau})_{\tau \in \Sigma} \in {\mathbb{Z}}^{\Sigma}$, write $t^{\mathbf{k}} \coloneqq \prod_{\tau \in \Sigma} t^{k_{\tau}} \in {\mathcal{R}}_L \left[ \tfrac{1}{t} \right]$. 
We calculate bases that give each triangulation associated with the elements of the Weyl group and transformation matrices between them. 

Let $D$ be a $(\varphi, \Gamma)$-module over ${\mathcal{R}}_L$ of rank $m$. We assume that $D$ is crystalline and $D_{\crys} (D)$ has $m$ distinct Frobenius eigenvalues. 
Then fixing a refinement ${\mathcal{F}}_{\bullet} = {\mathcal{F}}_{\bullet, \text{id}}$ of $D_{\crys} (D)$, we associate the elements $w \in W(\GL_m, T_m) \cong {\mathfrak{S}}_m$ of the Weyl group with a refinement ${\mathcal{F}}_{w, \bullet}$ of $D_{\crys} (D)$ and write ${\mathcal{P}}_w$ for the corresponding triangulation of $D$ by Proposition \ref{G_Berger_corr}. Write $\mathcal{P} = {\mathcal{P}}_{\text{id}}$. We may take a sequence $\varphi_1, \ldots, \varphi_m$ of Frobenius eigenvalues of $D_{\crys} (D)$ and its corresponding eigenvector $x_1, \ldots, x_m$ such that the filtration ${\mathcal{F}}_{w, \bullet}$ of $D_{\crys} (D)$ is the same as 
\[
0 \subset (x_{w(1)}) \subset (x_{w(1)}, x_{w(2)}) \subset \cdots \subset (x_{w(1)}, \ldots, x_{w(m)}) = D_{\crys} (D), 
\]
where $w \colon \{ 1, \ldots, m \} \to \{ 1, \ldots, m \}$ is the corresponding function to $w \in W(\GL_m, T_m)$ by $W(\GL_m, T_m) \xrightarrow{\sim} {\mathfrak{S}}_m \xrightarrow{\sim} \Aut (\{ 1, \ldots, m \})$. Write $x_{w, i} \coloneqq x_{w(i)}$ and ${\mathcal{F}}_{w, i} \coloneqq (x_{w, 1}, \ldots, x_{w, i}) \subset D_{\crys} (D)$ for $i \in [0, m]$ and $w \in W(\GL_m, T_m)$. 

We also write $D_{w, i} \coloneqq \left( {\mathcal{F}}_{w, i} \otimes_{K \otimes_{{\mathbb{Q}}_p} L} {\mathcal{R}}_L \left[ \tfrac{1}{t} \right] \right) \cap D$. 
For elements $y_1, \ldots, y_m \in D$ and a triangulation $\mathcal{P}$ of $D$, we call $(y_1, \ldots, y_m)$ is a basis of $\mathcal{P}$ if the filtration 
\[
0 \subset (y_1) \subset (y_1, y_2) \subset \cdots \subset (y_1, \ldots, y_m) = D
\]
is a triangulation of $D$ and the same as $\mathcal{P}$. 
By the classical correspondence between triangulations and refinements, for $i \in [1, m]$ and $w \in W(\GL_m, T_m)$, we have the following commutative diagram: 
\[\xymatrix{
0 \ar[r] & \left( D_{w, i-1} \left[ \tfrac{1}{t} \right] \right)^{\Gamma} \ar[r] \ar@{=}[d] & \left( D_{w, i} \left[ \tfrac{1}{t} \right] \right)^{\Gamma} \ar[r] \ar@{=}[d] & \left( (D_{w, i} / D_{w, i-1}) \left[ \tfrac{1}{t} \right] \right)^{\Gamma} \ar[r] \ar@{=}[d] & 0 \\ 
0 \ar[r] & D_{\crys} (D_{w, i-1}) \ar[r] \ar@{^{(}->}[d] & D_{\crys} (D_{w, i}) \ar[r] \ar@{^{(}->}[d] & D_{\crys} (D_{w, i} / D_{w, i-1}) \ar[r] \ar@{^{(}->}[d] & 0 \\ 
0 \ar[r] & D_{w, i-1} \left[ \tfrac{1}{t} \right] \ar[r] & D_{w, i} \left[ \tfrac{1}{t} \right] \ar[r] & (D_{w, i} / D_{w, i-1}) \left[ \tfrac{1}{t} \right] \ar[r] & 0 \\ 
0 \ar[r] & D_{w, i-1} \ar[r] \ar@{^{(}->}[u] & D_{w, i} \ar[r] \ar@{^{(}->}[u] & D_{w, i} / D_{w, i-1} \ar[r] \ar@{^{(}->}[u] & 0, 
}\]
where the horizontal maps form exact sequences, the vertical maps between the second row and the third row are the inclusions $- \mapsto \left( - \otimes_{K_0 \otimes_{{\mathbb{Q}}_p} L} {\mathcal{R}}_L \left[ \tfrac{1}{t} \right] \right)$ and the bottom vertical maps are the inclusions $- \mapsto \left( - \otimes_{{\mathcal{R}}_L} {\mathcal{R}}_L \left[ \tfrac{1}{t} \right] \right)$. 

Moreover, the bottom right module $D_{w, i} / D_{w, i-1}$ is the rank $1$ $(\varphi, \Gamma)$-module that is generated by $(t^{{\mathbf{k}}_i} x_{w, i} \ \text{mod} \ D_{i-1})$ over ${\mathcal{R}}_L$ for some $({\mathbf{k}}_1, \ldots, {\mathbf{k}}_m) \in ({\mathbb{Z}}^{\Sigma})^m$. 
Let $U_m^- \subset \GL_m$ be the subgroup of lower triangular unipotent matrices. 
Then there exists a matrix ${\mathbf{u}}_w = (u_{w, i, j})_{i, j} \in U_m^- \left( {\mathcal{R}}_L \left[ \tfrac{1}{t} \right] \right)$ and a unique $\mathbf{k} = ({\mathbf{k}}_1, \ldots, {\mathbf{k}}_m) \in ({\mathbb{Z}}^{\Sigma})^m$ such that ${\mathbf{k}}_i \in {\mathbb{Z}}^{\Sigma}$ and ${\mathbf{u}}_w {\mathbf{t}}_w {^t}(x_1, \ldots, x_m)$ is a basis of ${\mathcal{P}}_w$ when we write ${\mathbf{t}}_w \coloneqq \diag (t^{{\mathbf{k}}_1}, \ldots, t^{{\mathbf{k}}_m})$ by Lemma \ref{Berger_parameter}. 

For $w, w' \in W(\GL_m, T_m)$, write ${\mathbf{a}}_{w, w'} = (a_{w, w', i, j})_{i, j} \in \GL_m \left( {\mathcal{R}}_L \left[ \tfrac{1}{t} \right] \right)$ for a transformation matrix from the basis of ${\mathcal{P}}_w$ to that of ${\mathcal{P}}_{w'}$. Note that ${\mathbf{u}}_w$ and ${\mathbf{a}}_{w, w'}$ is not unique in $\GL_m \left( {\mathcal{R}}_L \left[ \tfrac{1}{t} \right] \right)$ but is unique in $U_m^- ({\mathcal{R}}_L) \backslash \GL_m \left( {\mathcal{R}}_L \left[ \tfrac{1}{t} \right] \right) / U_m^- ({\mathcal{R}}_L)$. A matrix ${\mathbf{a}}_{w, w'}$ is written in the form ${\mathbf{u}}_{w'} {\mathbf{t}}_{w'} w {\mathbf{t}}_w^{-1} {\mathbf{u}}_w^{-1}$ where $w \in W(\GL_m, T_m) = N_{\GL_m} (T_m) / C_{\GL_m} (T_m)$ is the matrix associated with the element of the Weyl group. 
For $1 \leq i < j \leq m$, we call an element $w \in W(\GL_m, T_m) \cong {\mathfrak{S}}_m$ is \emph{of type $(i, j)$} if $w(l) = l$ for $l < i$ or $j < l$, $w(i) = j$ and $w(j) = i$ as an element of ${\mathfrak{S}}_m = \Aut (\{ 1, \ldots, m \})$. 
Let $M_{i \times j} (A)$ be the additive abelian group of $(i \times j)$-matrices with coefficients in $A$. 

\begin{lemma}
\label{transform_matrices}
For an element $\sigma \in W(\GL_m, T_m)$ of type $(i, j)$ and $w, w' \in W(\GL_m, T_m)$ such that $w \sigma = w'$, we may take ${\mathbf{a}}_{w, w'}$ as the form 
\[\begin{pmatrix}
I_{i-1} & 0_{(i-1) \times (j-i+1)} & 0_{(i-1) \times (m-j)} \\ 
{\mathbf{n}}_{w, w'} & {\mathbf{a}}'_{w, w'} & 0_{(j-i+1) \times (m-j)} \\ 
0_{(m-j) \times (i-1)} & 0_{(m-j) \times (j-i+1)} & I_{m-j} \\ 
\end{pmatrix}\]
where ${\mathbf{a'}}_{w, w'} \in \GL_{j-i+1} ({\mathcal{R}}_L)$ and ${\mathbf{n}}_{w, w'} \in {\text{M}}_{(j-i+1) \times (i-1)} ({\mathcal{R}}_L)$. 
Moreover, writing 
\[
(z^{{\mathbf{k}}_{w, 1}} \unr(\varphi_{w, 1}), \ldots, z^{{\mathbf{k}}_{w, m}} \unr(\varphi_{w, m})), \ (z^{{\mathbf{k}}_{w', 1}} \unr(\varphi_{w', 1}), \ldots, z^{{\mathbf{k}}_{w', m}} \unr(\varphi_{w, m}))
\]
for the parameters associated to the triangulations ${\mathcal{P}}_w$, ${\mathcal{P}}_{w'}$, and when we write ${\mathbf{a'}}_{w, w'} = (a'_{w, w', \alpha, \beta})_{\alpha, \beta \in [1, j-i+1]}$, we have $a'_{w, w', 1, j-i+1} = t^{{\mathbf{k}}_{w, i} - {\mathbf{k}}_{w', j}}$.  
\end{lemma}

\begin{proof}
By the classical correspondence between triangulations and refinements, $D_{w, l} = D_{w', l}$ for $l < i$ or $j < l$. So the first part is shown. 

For the last part, since matrices ${\mathbf{u}}_w, {\mathbf{t}}_w, w$ are in $U_m^- \diag (\GL_{i-1}, \GL_{j-i+1}, \GL_{m-j})$, it suffices to calculate the middle $(j-i+1) \times (j-i+1)$-matrix ${\mathbf{a}}'_{w, w'}$. Write ${\mathbf{u}}'_w, {\mathbf{t}}'_w, w'$ for the factors of the middle $(j-i+1) \times (j-i+1)$-matrices of ${\mathbf{u}}_w, {\mathbf{t}}_w, w$. We have 
\begin{align*}
{\mathbf{a}}'_{w, w'} &= {\mathbf{u}}'_{w'} {\mathbf{t}}'_{w'} w' {\mathbf{t}}_w^{\prime -1} {\mathbf{u}}_w^{\prime -1} \\ 
&= {\mathbf{u}}'_{w'} \begin{pmatrix} 0 & 0_{1 \times (m-2)} & t^{{\mathbf{k}}_{w, i} - {\mathbf{k}}_{w', j}} \\ 0_{(m-2) \times 1} & \ast_{(m-2) \times (m-2)} & 0_{(m-2) \times 1} \\ t^{{\mathbf{k}}_{w, j} - {\mathbf{k}}_{w', i}} & 0_{1 \times (m-2)} & 0 \\ \end{pmatrix} {\mathbf{u}}_w^{\prime -1} \\ 
&= \begin{pmatrix} 0 & 0_{1 \times (m-2)} & t^{{\mathbf{k}}_{w, i} - {\mathbf{k}}_{w', j}} \\ 0_{(m-2) \times 1} & \ast_{(m-2) \times (m-2)} & \ast_{(m-2) \times 1} \\ t^{{\mathbf{k}}_{w, j} - {\mathbf{k}}_{w', i}} & \ast_{1 \times (m-2)} & \ast \\ \end{pmatrix} {\mathbf{u}}_w^{\prime -1} \\ 
&= \begin{pmatrix} \ast & \ast_{1 \times (m-2)} & t^{{\mathbf{k}}_{w, i} - {\mathbf{k}}_{w', j}} \\ \ast_{(m-2) \times 1} & \ast_{(m-2) \times (m-2)} & \ast_{(m-2) \times 1} \\ \ast & \ast_{1 \times (m-2)} & \ast \\ \end{pmatrix}.
\end{align*}
\end{proof}

\begin{definition}
\begin{enumerate}
\item We write ${\mathcal{F}}_{K, \tau} = (D_{\crys} (D) \otimes_{K_0} K) \otimes_{K \otimes_{{\mathbb{Q}}_p} L, \tau} L$. If $D$ is a crystalline $(\varphi, \Gamma)$-module of rank $m$, its \emph{Hodge-Tate type} is the sequences of integers $(k_{\tau, 1} \geq \cdots \geq k_{\tau, m})_{\tau \in \Sigma}$ where the $k_{\tau, i}$ are the integers $s$ such that $\gr^{-s} {\mathcal{F}}_{K, \tau} \neq 0$, counted with multiplicity, where the multiplicity of $s$ is defined as the dimension $\dim_L \gr^{-s} {\mathcal{F}}_{K, \tau}$. 
\item A \emph{Hodge-Tate type} is an element $\mathbf{k} = ({\mathbf{k}}_1, \ldots, {\mathbf{k}}_m) \in ({\mathbb{Z}}^{\Sigma})^m$ where ${\mathbf{k}}_i = (k_{\tau, 1})_{\tau \in \Sigma}$ and $(k_{\tau, 1}, \ldots, k_{\tau, m})$ is a decreasing sequence of integers for each $\tau \in \Sigma$. We say that Hodge-Tate type $\mathbf{k}$ is \emph{regular} if all these sequences of integers are strictly decreasing. 
\item We say that a crystalline $(\varphi, \Gamma)$-module $D$ of rank $m$ is \emph{$\varphi$-generic} if the Frobenius $\Phi = \phi^{[K_0 \colon {\mathbb{Q}}_p]}$ on $D_{\crys} (D)$ has $m$ distinct eigenvalues $(\varphi_1, \ldots, \varphi_m)$ in $L$ and $\varphi_i \varphi_j^{-1} \neq p^{[K_0 \colon {\mathbb{Q}}_p]}$ for $i \neq j$.
\item We say that a crystalline $(\varphi, \Gamma)$-module $D$ with a regular Hodge-Tate type is \emph{noncritical} if a trianguline parameter associated with each refinement of $D_{\crys} (D)$ is noncritical. 
\item We say that a crystalline $(\varphi, \Gamma)$-module $D$ is \emph{benign} if $D$ has a regular Hodge-Tate type, noncritical and $\varphi$-generic. 
\end{enumerate}
\end{definition}

\begin{remark}
In \cite[Definition 5.5]{Iye19}, $\varphi$-genericity is defined as the definition above excluding only the last condition. However, we use the definition in \cite[\S 3.3]{HMS18} for the simplification of the notation. 
\end{remark}


\begin{lemma}
\label{H_surjectivity}
Let $D$ be a rank $m$ trianguline $(\varphi, \Gamma)$-module and $\mathcal{P}$ be a triangulation of $D$ that has the parameter $(\delta_1, \ldots, \delta_m)$ such that $\delta_i = z^{{\mathbf{k}}_i} \unr (\varphi_i)$ for $\mathbf{k} = ({\mathbf{k}}_1, \ldots, {\mathbf{k}}_m) \in ({\mathbb{Z}}^{\Sigma})^m$ and $i \in [1, m]$, $D'$ be a rank $m$ $(\varphi, \Gamma)$-submodule of $D$, ${\mathcal{P}}'$ be the triangulation of $D'$ induced from $\mathcal{P}$ and $(\delta'_1, \ldots, \delta'_m)$ be the parameter of ${\mathcal{P}}'$ such that $\delta'_i = z^{{\mathbf{k}}'_i} \unr (\varphi_i)$ for ${\mathbf{k}}' = ({\mathbf{k}}'_1, \ldots, {\mathbf{k}}'_m) \in ({\mathbb{Z}}^{\Sigma})^m$ and $i \in [1, m]$. 

Write $S(D, \mathcal{P}) \subset \{ 1, \ldots, m \}$ for the subset such that $\delta_l \neq \delta'_l$ if and only if $l \in \{ 1, \ldots, m \}$. 
For each $i \in S(D, \mathcal{P})$, assume that $\varphi_i \not\in \{1, p^{[K_0 \colon {\mathbb{Q}}_p]} \}$ and $\max (k_{i, \tau}, 1 - k'_{i, \tau}) > 0$ for all $\tau \in \Sigma$. Then the induced map $H^1 (D') \to H^1 (D)$ of the first cohomology modules is surjective. 
\end{lemma}

\begin{proof}
Let $D_i / D_{i-1}$ and $(D_i \cap D')/(D_{i-1} \cap D')$ for $i \in [1, m]$. Since $D$ is trianguline with parameter $(\delta_1, \ldots, \delta_m)$ and $D'$ is a submodule of $D$, it suffices to show that $H^1 ({\mathcal{R}}_L (\delta'_i)) \to H^1 ({\mathcal{R}}_L (\delta_i))$ is surjective. 

In the case of $i \notin S(D, \mathcal{P})$, we obtain $\dim_L H^1 ({\mathcal{R}}_L (z^{{\mathbf{k}}_i} \delta)) = \dim_L H^1 ({\mathcal{R}}_L (\delta))$ by the assumptions and Lemma \ref{phiGamma_L_coh}. 
Then the surjectivity of the map $H^1 ({\mathcal{R}}_L (\delta'_i)) \to H^1 ({\mathcal{R}}_L (\delta_i))$ is reduced to show that $H^0 ({\mathcal{R}}_L (\delta_i) / t^{{\mathbf{k}}''_i} {\mathcal{R}}_L (\delta_i)) = 0$ for $i \in S(D, \mathcal{P})$ and ${\mathbf{k}}'' \coloneqq \mathbf{k} - {\mathbf{k}}' \in {\mathbb{Z}}_{>0}^{\Sigma}$ by the long exact sequence. It is reduced to show $H^0 (t^{{\mathbf{k}}'''}{\mathcal{R}}_L (\delta_i) / t_{\tau} t^{{\mathbf{k}}'''}{\mathcal{R}}_L (\delta_i)) = 0$ for all $\tau \in \Sigma$ and ${\mathbf{k}}''' \in {\mathbb{Z}}_{\geq 0}^{\Sigma}$ such that ${\mathbf{k}}_{i, \tau} + {\mathbf{k}}'''_{i, \tau} \neq 0$, as well by d\'evissage. This follows from that $k_{i, \tau} > 0$ for all $\tau \in \Sigma$ and \cite[Lemma 6.2.12]{KPX12}. 
\end{proof}

There is an element $w \in W(\GL_m, T_m)$ such that the triangulation ${\mathcal{P}}_w$ has a noncritical parameter. From now, we replace the triangulation $\mathcal{P}$ with the ${\mathcal{P}}_w$. 
We take a basis $(y_1, \ldots, y_m)$, $y_i \in D$ for $i \in [1, m]$ of ${\mathcal{P}}_{\text{id}}$. Let $(e_{\alpha, \beta})_{\alpha, \beta \in [1, m]}$ be the associated standard basis of $\ad (D)$ with the basis $(y_1, \ldots, y_m)$ of $D$. Let ${\mathbf{a}}_w \coloneqq {\mathbf{a}}_{\text{id}, w}$. 
Then we may translate $\ad_{{\mathcal{P}}_w} (D)$ for $w \in W(\GL_m, T_m)$ in terms of the standard basis. We have 
\begin{align*}
\ad_{{\mathcal{P}}} (D) &= \langle e_{\alpha, \beta} \ | \ 1 \leq \beta < \alpha \leq m \rangle \eqqcolon {\mathfrak{b}}_{\mathcal{P}}^- \subset \ad (D), \\ 
\ad_{{\mathcal{P}}_{w}} (D) &= {\mathbf{a}}_w ({\mathfrak{b}}_{\mathcal{P}}^-). 
\end{align*}

Let ${\sigma}_{i, j} \in W(\GL_m, T_m) \cong {\mathfrak{S}}_m$ be the $(i, j)$-transposition for $1 \leq i, j \leq m$, $i \neq j$ and $W' \subset W(\GL_m, T_m)$ be the subset of $\sigma_{i, j}$. For $\mathbf{a}, \mathbf{b} \in \ad (D)$, we write $\ad (\mathbf{a}) (\mathbf{b}) \coloneqq \mathbf{a} \mathbf{b} {\mathbf{a}}^{-1}$. 

\begin{lemma}
\label{ad_surjectivity_GL}
The saturation of the image of the map 
\[
\bigoplus_{w \in W(\GL_m, T_m)} \ad_{{\mathcal{P}}_w} (D) \to \ad (D)
\]
of $(\varphi, \Gamma)$-modules over ${\mathcal{R}}_L$ is $\ad (D)$. Moreover, this embedding of $(\varphi, \Gamma)$-modules satisfies the conditions in Lemma \ref{H_surjectivity}.
\end{lemma}

\begin{proof}
We prove the assertion for the map 
\[
\bigoplus_{w \in W'} \ad_{{\mathcal{P}}_w} (D) \to \ad (D)
\]
instead. 
As ${\mathbf{a}}'_{\sigma_{i, j}} \in \GL_{j-i+1} ({\mathcal{R}}_L)$, there exists $h = h(\tau, i, j) \in [0, j-i]$ such that ${\mathbf{a}}_{i+h, j} = {\mathbf{a}}'_{h+1, j-i+1} \in {\mathcal{R}}_L \backslash t_{\tau} {\mathcal{R}}_L$ for each $\tau \in \Sigma$ and fix such $h(\tau, i, j)$ for $\tau \in \Sigma$ and $1 \leq i < j \leq m$. The element $\ad ({\mathbf{a}}_{\sigma_{i, j}}) (e_{j, i + h_{\tau}})$ is of the form 
\[
\theta e_{i, j} \ \text{mod} \ \langle e_{\alpha, \beta} \ | \ 1 \leq \alpha \leq i, \ j \leq \beta \leq i, \ (\alpha, \beta) \neq (i, j) \rangle 
\]
where $\theta \in (t_{\tau}^{{\mathbf{k}}_{\text{id}, i} - {\mathbf{k}}_{\sigma_{i, j}, j}} {\mathcal{R}}) \backslash (t_{\tau}^{{\mathbf{k}}_{\text{id}, i} - {\mathbf{k}}_{\sigma_{i, j}, j} + 1} {\mathcal{R}})$ by Lemma \ref{transform_matrices}. 

Then we can show that the saturation of the composition of the map in the assertion and the projection from $D$ into 
\[
{\mathfrak{b}}_{\mathcal{P}, l}^- \coloneqq \langle e_{\alpha, \beta} \ | \ 1 \leq \alpha, \ \beta \leq m, \ \beta - \alpha \leq l \rangle
\]
is ${\mathfrak{b}}_{\mathcal{P}, l}^-$ for $1 \leq l \leq m-1$ by induction. When $l = m-1$, it is just as the assertion of the lemma. 
The last part of the Lemma follows by the fact that $(t_{\tau}) + (t_{\tau'}) = {\mathcal{R}}_L$ for $\tau \neq \tau'$ and Corollary \ref{subphiGamma_rank1}. 
\end{proof}

We reprove \cite[Theorem 3.11]{HMS18}. 

\begin{proposition}
\label{surjectivity_GL}
Let $D$ be a $\varphi$-generic rank $m$ crystalline $(\varphi, \Gamma)$-module of regular Hodge-Tate type. Then the map 
\[
\bigoplus_{w \in W(\GL_m, T_m)} TX_{D, {\mathcal{P}}_w} \to TX_D
\]
is surjective. 
\end{proposition}

\begin{proof}
Let $D'$ be the image of the map  
\[
\bigoplus_{w \in W(\GL_m, T_m)} \ad_{{\mathcal{P}}_w} (D) \to \ad(D)
\]
in Lemma \ref{ad_surjectivity_GL}. Since the inclusion $D' \subset \ad(D)$ satisfies the assumption of Lemma \ref{H_surjectivity} by the construction, it follows that $H^1 (D') \cong H^1 (\ad(D))$. 
We also obtain that $H^2 (E) = 0$ for all $(\varphi, \Gamma)$-submodules $E$ of the trianguline $(\varphi, \Gamma)$-module $\ad (D)$ because of the $H^2$ of the trianguline parameter of $\ad (D)$ is $0$ by the $\varphi$-genericity of $D$ and the right exactness of $H^2$. 

Let 
\[
D'' \coloneqq \ker \left( \bigoplus_{w \in W(\GL_m, T_m)} \ad_{{\mathcal{P}}_w} (D) \to \ad(D) \right).
\]
It follows that the map 
\[
H^1 \left( \bigoplus_{w \in W(\GL_m, T_m)} \ad_{{\mathcal{P}}_w} (D) \right) \to H^1 (\ad(D))
\]
is surjective by the long exact sequence associated with the exact sequence 
\[
0 \to D'' \to \bigoplus_{w \in W(\GL_m, T_m)} \ad_{{\mathcal{P}}_w} (D) \to D' \to 0
\]
of $(\varphi, \Gamma)$-modules over ${\mathcal{R}}_L$ and $H^2 (D'') = 0$. The result follows by Lemma \ref{tangent_H1adD}, (1). 
\end{proof}

From here we consider the case of $G = \GSp_{2n}$ over $k = K$. We write $B \subset \GSp_{2n}$ for the standard Borel subgroup and $T \subset B$ for the maximal split torus. 
Let $D$ be a frameable $(\varphi, \Gamma)$-module with $\GSp_{2n}$-structure over ${\mathcal{R}}_L$ and $\mathcal{P}$ be a triangulation of $D$. 

Let $r_{\std} \colon \GSp_{2n} \to \GL_{2n}$ denote the standard representation of $\GSp_{2n}$. It is the embedding of regular type. We write $\widetilde{D} \coloneqq r_{{\std}, {\ast}} (D)$. 
We assume that $D$ is crystalline, $\widetilde{D}$ is $\varphi$-generic with regular Hodge-Tate type and there exists a refinement of $D_{\crys} (D)$. 
Moreover, we assume that there is a triangulation ${\mathcal{P}}$ which has a noncritical parameter and fix such ${\mathcal{P}}_{\text{id}} = {\mathcal{P}}$.

For the corresponding refinement $\mathfrak{P} = {\mathfrak{P}}_{\text{id}}$ of $D_{\crys} (D)$ to ${\mathcal{P}}$, we associate the elements $w \in W(\GSp_{2n}, T)$ of the Weyl group with a refinement ${\mathfrak{P}}_w$ of $D_{\crys} (D)$ and write ${\mathcal{P}}_w$ for the corresponding triangulation of $D$ by Proposition \ref{G_Berger_corr}. 
We also write ${\mathcal{P}}_w$ for the induced triangulation of $\widetilde{D}$ from the triangulation ${\mathcal{P}}_w$ of $D$. 
It follows that $\ad_{\mathcal{P}_w} (D) = \ad (D) \cap \ad_{\mathcal{P}_w} (\widetilde{D})$ by Proposition \ref{tensor_auto_represent}, (3) and $\GSp_{2n} \cap B_{2n} = B$. 

Since the inclusion $\GSp_{2n} \hookrightarrow \GL_{2n}$ induces the injective map $W(\GSp_{2n}, T) \hookrightarrow W(\GL_{2n}, T_{2n})$ of the Weyl groups, we have to show Lemma \ref{ad_surjectivity_GL} restricting $W(\GL_{2n}, T_{2n})$ to $W(\GSp_{2n}, T)$ for obtaining a result of Proposition \ref{surjectivity_GL} in case of $\GSp_{2n}$. 

\begin{lemma}
\label{ad_surjectivity_GSp}
The saturation of the image of the map 
\[
\bigoplus_{w \in W(\GSp_{2n}, T)} \ad_{{\mathcal{P}}_w} (D) \to \ad (D)
\]
of $(\varphi, \Gamma)$-modules over ${\mathcal{R}}_L$ is $\ad (D)$. Moreover, this embedding of $(\varphi, \Gamma)$-modules satisfies the conditions in Lemma \ref{H_surjectivity}.
\end{lemma}

\begin{proof}
By the above argument, the map in the assertion is the same as 
\[
\bigoplus_{w \in W(\GSp_{2n}, T) \subset W(\GL_{2n}, T_{2n})} (\ad (D) \cap \ad_{{\mathcal{P}}_w} (\widetilde{D})) \to \ad (D). 
\]
Let $I_{\sat}$ be the image of the map. Calculating 
\[
\ad ({\mathbf{a}}_{\sigma_{i, j} \sigma_{2n+1-j, 2n+1-i}}) (e_{j, i} - e_{2n+1-i, 2n+1-j})
\]
for $1 \leq i < j \leq n$ instead of $\ad ({\mathbf{a}}_{\sigma_{i, j}}) (e_{j, i})$ for $1 \leq i < j \leq n$ in the proof of Lemma \ref{ad_surjectivity_GL}, we may view $I_{\sat}$ contains the opposite Siegel parabolic submodule 
\[
{\mathfrak{p}}^- \coloneqq \begin{pmatrix} \ast_{n \times n} & 0_{n \times n} \\ \ast_{n \times n} & \ast_{n \times n} \\ \end{pmatrix} \cap \ad (D). 
\]
Then, in a similar way considering $\ad ({\mathbf{a}}_{\sigma_{n, n+1}}) ({\mathfrak{p}}^-)$, we have 
\[
\langle e_{n, n+1}, e_{l, n+1} - e_{n, 2n+1-l} \ | \ l \in [1, n-1] \rangle \subset I_{\sat}.
\]
Finally considering 
\[
\ad ({\mathbf{a}}_{\sigma_{k, n} \sigma_{n+1, 2n+1-k}}) (\langle e_{n, n+1}, e_{l, n+1} - e_{n, 2n+1-l} \ | \ l \in [1, n-1] \rangle), 
\]
we have 
\[
\begin{pmatrix} 0_{n \times n} & \ast_{n \times n} \\ 0_{n \times n} & 0_{n \times n} \\ \end{pmatrix} \cap \ad (D) \subset I_{\sat}. 
\]
Consequently, we obtain $I_{\sat} = \ad (D)$. 
For the last part, additionally calculating the images of $\ad_{{\mathcal{P}}_{\sigma}} (D) \to \ad (D)$ for $\sigma = \sigma_{c, 2n+1-c}$, $c \in [1, n]$ as in Lemma \ref{ad_surjectivity_GL}, we can check that the embedding of $(\varphi, \Gamma)$-modules satisfies the conditions in Lemma \ref{H_surjectivity} by the constructions and the calculations above. 
\end{proof}

\begin{corollary}
\label{surjectivity_GSp}
Let $D$ be a crystalline $(\varphi, \Gamma)$-module with $\GSp_{2n}$-structure over ${\mathcal{R}}_L$ such that $\widetilde{D}$ is $\varphi$-generic with regular Hodge-Tate type and assume that there exists a refinement of $D_{\crys} (D)$. Then the map 
\[
\bigoplus_{w \in W(\GSp_{2n}, T)} TX_{D, {\mathcal{P}}_w} \to TX_D
\]
is surjective. 
\end{corollary}

\begin{proof}
The map of $(\varphi, \Gamma)$-modules with $\GSp_{2n}$-structures 
\[
\bigoplus_{w \in W(\GSp_{2n}, T)} \ad_{{\mathcal{P}}_w} (D) \to \ad (D)
\]
in the Lemma \ref{ad_surjectivity_GSp} satisfies the condition of Lemma \ref{H_surjectivity} by its construction. Thus the result follows by Lemma \ref{H_surjectivity}. 
\end{proof}


\section{Locally irreducibility and accumulation property} 

In this section, we prove locally irreducibility and accumulation property of the trianguline deformation spaces for $\GSp_{2n}$. 

\subsection{Almost de Rham representations with $G$-structures}

Let us recall the notations in \cite[\S 3]{BHS17}. Let $B_{\dR}$, $B_{\dR}^+$ be the usual de Rham period rings of Fontaine with the ${\mathcal{G}}_K$-action and the natural topology and all finite type modules over these rings are endowed with the natural topology (cf.\ \cite[\S 3.2]{Fon04}). Let $\Rep_{B_{\dR}} ({\mathcal{G}}_K)$ (resp.\ $\Rep_{B_{\dR}^+} ({\mathcal{G}}_K)$ denote the category of $B_{\dR}$-representations (resp.\ $B_{\dR}^+$-representations) of ${\mathcal{G}}_K$ (cf.\ \cite[\S 3]{Fon04}). 
Let $W$ be an object of $\Rep_{B_{\dR}} ({\mathcal{G}}_K)$. It follows that $W$ contains a $B_{\dR}^+$-lattice stable under ${\mathcal{G}}_K$. 
We say that $W$ is \emph{almost de Rham} if it contains a ${\mathcal{G}}_K$-stable $B_{\dR}^+$-lattice $W^+$ such that the Sen weights of the $C$-representation $W^+ / tW^+$ are all in $\mathbb{Z}$ (cf.\ \cite[\S 3.7]{Fon04}). 

Let $B_{\pdR}^+$ be the algebra $B_{\dR}^+ [\log (t)]$ and $B_{\pdR} \coloneqq B_{\dR} \otimes_{B_{\dR}^+} B_{\pdR}^+$. The group ${\mathcal{G}}_K$ acts on these algebras as $g(\log (t)) = \log (t) + \log (\epsilon (g))$. Let $\nu_{B_{\pdR}}$ be the unique $B_{\dR}$-derivation of $B_{\pdR}$ such that $\nu_{B_{\pdR}} (\log (t)) = -1$, which preserves the subalgebra $B_{\pdR}^+ \subset B_{\pdR}$. 
We set $D_{\pdR} (W) \coloneqq (B_{\pdR} \otimes_{B_{\dR}} W)^{{\mathcal{G}}_K}$, which is a finite dimensional $K$-vector space of dimension $\leq \dim_{B_{\dR}} W$. It follows that $W$ is almost de Rham if and only if $\dim_K D_{\pdR} (W) = \dim_{B_{\dR}} W$ by \cite[Th\'eor\`em 4.1 (2)]{Fon04}. 

We say that $W$ is \emph{de Rham} if $\dim_K W^{{\mathcal{G}}_K} = \dim_{B_{\dR}} W$. If $W$ is de Rham, $W$ is also almost de Rham. We write $\Rep_{\pdR} ({\mathcal{G}}_K)$ for the category of almost de Rham $B_{\dR}$-representations, which is a Tannakian subcategory of $\Rep_{B_{\dR}} ({\mathcal{G}}_K)$ and is stable under kernel, cokernel and extensions (cf.\ \cite[\S 3.7]{Fon04}). 

Let $A$ be a finite dimensional ${\mathbb{Q}}_p$-algebra. An \emph{$A \otimes_{{\mathbb{Q}}_p} B_{\dR}$-representations} of ${\mathcal{G}}_K$ is a $B_{\dR}$-representation $W$ of ${\mathcal{G}}_K$ together with a morphism of ${\mathbb{Q}}_p$-algebras $A \to \End_{\Rep_{B_{\dR}} ({\mathcal{G}}_K)} (W)$ which makes $W$ a finite free $A \otimes_{{\mathbb{Q}}_p} B_{\dR}$-module. Let $\Rep_{A \otimes_{{\mathbb{Q}}_p} B_{\dR}} ({\mathcal{G}}_K)$ be the category of $A \otimes_{{\mathbb{Q}}_p} B_{\dR}$-representations of ${\mathcal{G}}_K$. We similarly define \emph{$A \otimes_{{\mathbb{Q}}_p} B_{\dR}^+$-representations} of ${\mathcal{G}}_K$ and their category $\Rep_{A \otimes_{{\mathbb{Q}}_p} B_{\dR}^+} ({\mathcal{G}}_K)$. We say that an $A \otimes_{{\mathbb{Q}}_p} B_{\dR}$-representation of ${\mathcal{G}}_K$ is \emph{almost de Rham} if the underlying $B_{\dR}$-representation is. Write $\Rep_{\pdR, A} ({\mathcal{G}}_K)$ for the category of almost de Rham representations of ${\mathcal{G}}_K$. 

Let $(G, B, T)$ be a split reductive group of regular type over $K$. For an abelian tensor category $\mathcal{A}$, let $G$-$\mathcal{A}$ denote the category of exact tensor functors $\omega \colon \Rep_K (G) \to \mathcal{A}$. For a exact tensor functor $F \colon \mathcal{A} \to \mathcal{A}'$ between abelian tensor categories, we also write $F \colon G\text{-}\mathcal{A} \to G\text{-}\mathcal{A}'$ the induced functor. An \emph{$A \otimes_{{\mathbb{Q}}_p} B_{\dR}$-representation with $G$-structure} of ${\mathcal{G}}_K$ is an object of $G$-$\Rep_{A \otimes_{{\mathbb{Q}}_p} B_{\dR}} ({\mathcal{G}}_K)$ and a \emph{$A \otimes_{{\mathbb{Q}}_p} K$-representation with $G$-structure} of ${\mathbb{G}}_a$ is an object of $G$-$\Rep_{A \otimes_{{\mathbb{Q}}_p} K} ({\mathbb{G}}_a)$. We say that $\omega \in \text{ob} (G\text{-}\Rep_{A \otimes_{{\mathbb{Q}}_p} B_{\dR}} ({\mathcal{G}}_K))$ is almost de Rham if $\omega (V)$ is almost de Rham for all $V \in \Rep_K (G)$.  Write $G$-$\Rep_{\pdR, A} ({\mathcal{G}}_K)$ for the category of almost de Rham $A \otimes_{{\mathbb{Q}}_p} B_{\dR}$-representations with $G$-structures of ${\mathcal{G}}_K$. 

\begin{lemma}
\label{DpdR_eqv}
The functor $D_{\pdR}$ induces an equivalence of categories between $G$-$\Rep_{\pdR, A} ({\mathcal{G}}_K)$ and $G$-$\Rep_{A \otimes_{{\mathbb{Q}}_p} K} ({\mathbb{G}}_a)$. 
\end{lemma}

\begin{proof}
This follows from \cite[Lemma 3.1.4]{BHS17} and Tannakian formalism. 
\end{proof}

\subsection{Trianguline deformation functors of almost de Rham representations with $G$-structures}

Fix an almost $L \otimes_{{\mathbb{Q}}_p} B_{\dR}$-representation with $G$-structure $W$. 
A \emph{triangulation} ${\mathcal{P}}_W$ of $W$ is a $B$-torsor $\Rep_K (B) \to \Rep_{L \otimes_{{\mathbb{Q}}_p} B_{\dR}} ({\mathcal{G}}_K)$ such that the composition with $\Rep_K (G) \to \Rep_K (B)$ is the same as the $W$. 
We assume that the underlying $G$-torsor over $L \otimes_{{\mathbb{Q}}_p} K$ of $D_{\pdR} (W)$ is trivial replacing $K$, $L$ with sufficiently large finite extensions if necessary. Let $\omega_{\text{st}} \colon \Rep_K (G) \to \Bun (L \otimes_{{\mathbb{Q}}_p} K)$ be the standard fiber functor. We also fix an isomorphism $\alpha \colon \omega_{\text{st}} \xrightarrow{\sim} D_{\pdR} (W)$. 

\begin{definition}
\begin{enumerate}
\item Let $X_W$ be the groupoid over ${\mathcal{C}}_L$ as follows: 
\begin{itemize}
\item The objects of $X_W (A)$ are pairs $(W_A, \iota_A)$ where $A$ is an object of ${\mathcal{C}}_L$, $W_A$ is an object of $G$-$\Rep_{\pdR, A} ({\mathcal{G}}_K)$ and $\iota_A \colon W_A \times_A L \xrightarrow{\sim} W$. 
\item A morphism $(W_A, \iota_A) \to (W_{A'}, \iota_{A'})$ is a map $A \to A'$ in ${\mathcal{C}}_L$ and an isomorphism $W_A \otimes_A A' \xrightarrow{\sim} W_{A'}$ compatible with the morphisms $\iota_A$, $\iota_{A'}$. 
\end{itemize}
\item Let $X_W^{\Box}$ be the groupoid over ${\mathcal{C}}_L$ as follows: 
\begin{itemize}
\item The objects of $X_W^{\Box} (A)$ are $(W_A, \iota_A, \alpha_A)$ with $(W_A, \iota_A)$ an object of $X_W (A)$ and $\alpha_A \colon \omega_{\text{st}} \otimes_{{\mathbb{Q}}_p} A$ such that the following diagram commutes: 
\[\xymatrix{
\omega_{\text{st}} \otimes_{{\mathbb{Q}}_p} A \ar[r]^-{1 \otimes \alpha_A} \ar@{=}[d] & D_{\pdR} (W_A) \otimes_A L \ar[d]^{\cong} \\ 
\omega_{\text{st}} \otimes_{{\mathbb{Q}}_p} A \ar[r]^-{\alpha} & D_{\pdR} (W). 
}\]
\item A morphism $(W_A, \iota_A, \alpha_A) \to (W_{A'}, \iota_{A'}, \alpha_{A'})$ is a morphism $(W_A, \iota_A) \to (W_{A'}, \iota_{A'})$ in $X_W$ such that the following diagram commutes: 
\[\xymatrix{
(\omega_{\text{st}} \otimes_{{\mathbb{Q}}_p} A) \otimes_A A' \ar[r]^-{1 \otimes \alpha_A} \ar@{=}[d] & D_{\pdR} (W_A) \otimes_A A' \ar[d]^{\cong} \\ 
\omega_{\text{st}} \otimes_{{\mathbb{Q}}_p} A' \ar[r]^-{\alpha_{A'}} & D_{\pdR} (W_{A'}). 
}\]
\end{itemize}
\item For a fixed triangulation ${\mathcal{P}}_W$ of $W$, let $X_{W, {\mathcal{P}}_W}$ be the groupoid over ${\mathcal{C}}_L$ as follows: 
\begin{itemize}
\item The objects of $X_{W, {\mathcal{P}}_W} (A)$ are $(W_A, {\mathcal{P}}_{W, A}, \iota_A)$ where $(W_A, \iota_A)$ is an object of $X_W (A)$ and ${\mathcal{P}}_{W, A}$ is a triangulation of $W_A$ such that $\iota_A$ induces an isomorphism ${\mathcal{P}}_{W, A} \otimes_A L \xrightarrow{\sim} {\mathcal{P}}_W$. 
\item The morphisms are the morphisms in $X_W$ which induce an isomorphism ${\mathcal{P}}_{W, A} \otimes_A A' \xrightarrow{\sim} {\mathcal{P}}_{W, A'}$. 
\end{itemize}
\item For a fixed triangulation ${\mathcal{P}}_W$ of $W$, let $X_{W, {\mathcal{P}}_W}^{\Box} \coloneqq X_{W, {\mathcal{P}}_W} \times_{X_W} X_W^{\Box}$. 
\end{enumerate}
\end{definition}

We set an algebraic group $\underline{G}$ over $L$ as 
\[
\underline{G} \coloneqq \Spec L \times_{\Spec_{{\mathbb{Q}}_p}} \Res_{K / {\mathbb{Q}}_p} (\GL_{n, K}) \cong (\GL_{n, L})^{\Sigma}. 
\] 
Similarly we define $\underline{B}$, $\underline{T}$. Let $\mathfrak{g}$, $\mathfrak{b}$, $\mathfrak{t}$ be the Lie algebra over $L$ of respectively $\underline{G}$, $\underline{B}$, $\underline{T}$. 
For a finite ${\mathbb{Q}}_p$-algebra $A$, we write $\mathfrak{g} (A) \coloneqq \mathfrak{g} {\mathbb{Q}}_p A$ and do the same for $\mathfrak{b}$, $\mathfrak{t}$. 
Let $\widetilde{\mathfrak{g}}$ be the $L$-scheme defined by 
\[
\widetilde{\mathfrak{g}} \coloneqq \{ (g \underline{B}, \psi) \in (\underline{G} / \underline{B}) \times \mathfrak{g} | \ad (g^{-1}) \psi \in \mathfrak{b} \} \subset (\underline{G} / \underline{B}) \times \mathfrak{g}. 
\]
We write $q \colon \widetilde{\mathfrak{g}} \to \mathfrak{g}$ for the Grothendieck's simultaneous resolution of singularities, which is defined by $(g \underline{B}, \psi) \mapsto \psi$. 
Let $X$ be the $L$-scheme defined by 
\[
X \coloneqq \widetilde{\mathfrak{g}} \times_{\mathfrak{g}} \widetilde{\mathfrak{g}} = 
\{ (g_1 \underline{B}, g_2 \underline{B}, \psi) \in (\underline{G} / \underline{B}) \times (\underline{G} / \underline{B}) \times \mathfrak{g} | \ad (g_1^{-1}) \psi \in \mathfrak{b}, \ad (g_2^{-1}) \psi \in \mathfrak{b} \}, 
\]
where the fiber product is taken by the map $q$. 

\begin{lemma}
\label{GRepGa_nilp_eqv}
Let $E$ be a field of characteristic $0$ and $A$ be a finite dimensional ${\mathbb{Q}}_p$-algebra. There is a natural bijection between the set of homomorphisms of $A$-group schemes ${\mathbb{G}}_{a, A} \to G_A$ and the set of nilpotent elements of $\Lie_A (G_A)$. 

Moreover, the full subcategory of the objects, whose underlying torsor are trivial, of $G$-$\Rep_A ({\mathbb{G}}_a)$ is equivalent to the category of the pair $(G_A, \nu_A)$ where $\nu_A \in \Lie_A (G_A)$ is a nilpotent element. 
\end{lemma}

\begin{proof}
This follows from the preceding argument of \cite[Proposition 3.1.1]{BHS17}, the construction of a quasi-inverse of the equivalence in Lemma \ref{DpdR_eqv} (cf.\ \cite[Corollary 3.1.2]{BHS17}) and Lemma \ref{additive_reconst}. 
\end{proof}

Let $W$ be an almost de Rham $L \otimes_{{\mathbb{Q}}_p} B_{\dR}$-representation with $G$-structure and ${\mathcal{P}}_W$ be a triangulation of $W$. 
If $(W_A, \iota_A)$ is an object of $X_W (A)$ for $A \in {\mathcal{C}}_L$, let $\nu_{W_A} \coloneqq \nu_{D_{\pdR} (W_A)}$ denote the nilpotent endomorphism of $D_{\pdR} (W_A)$ corresponding by Lemma \ref{GRepGa_nilp_eqv} to the action of ${\mathbb{G}}_a$, which is given by Lemma \ref{DpdR_eqv}. 
If $(W_A, \iota_A, \alpha_A)$ is an object of $X_W^{\Box} (A)$, we define $N_{W_A} \in \mathfrak{g} (A)$ as the matrix of $\alpha_A^{-1} \circ \nu_{W_A} \circ \alpha_A$ in the canonical basis of $\omega_{\text{st}}$. In the case $A = L$, we simply write $N_W = N_{W_L}$. 

Let $\widehat{\mathfrak{g}}$ (resp.\ $\widehat{\mathfrak{t}}$) denote the completion of $\mathfrak{g}$ (resp.\ $\mathfrak{t}$) at the point $N_W \in \mathfrak{g} (L)$ (resp.\ $0 \in \mathfrak{t})$. 
Note that the set of triangulations of $\omega_{\text{st}, A} \colon \Rep_K (G) \to \Rep_{A \otimes_{{\mathbb{Q}}_p} B_{\dR}} ({\mathcal{G}}_K)$ is bijective to the subset of elements of $\{ gBg^{-1} \ | \ g \in \omega_{\text{st}}^{\text{geom}} (A \otimes_{{\mathbb{Q}}_p} B_{\dR}) \}$, where $\omega_{\text{st}}^{\text{geom}}$ is the underlying scheme of $\omega_{\text{st}}$ naturally isomorphic to the group scheme $G_{A \otimes_{{\mathbb{Q}}_p} B_{\dR}}$, which are invariant in the action of ${\mathcal{G}}_K$. 

It is also bijective to the subset of elements of $(G/B)(A \otimes_{{\mathbb{Q}}_p} B_{\dR})$, which are invariant in the action of ${\mathcal{G}}_K$, by $gBg^{-1} \mapsto gB$. 
let $\widehat{\widetilde{\mathfrak{g}}}$ denote the completion of $\widetilde{\mathfrak{g}}$ at the point $(\alpha^{-1} ({\mathcal{P}}_W), N_W) \in \widetilde{\mathfrak{g}} (L)$, where $\alpha^{-1} ({\mathcal{P}}_W) \in (\underline{G} / \underline{B})(L)$ is the object corresponding to the triangulation ${\mathcal{P}}_W$. 
If $A \in {\mathcal{C}}_L$, we similarly define $\alpha_A^{-1} ({\mathcal{P}}_{W, A})$ for triangulations ${\mathcal{P}}_W$ of $D_{\pdR} (D_A)$. 
For a groupoid $X$ over ${\mathcal{C}}_L$, let $|X|$ denote the associated functor with $X$. 

\begin{corollary}
\begin{enumerate}
\item The groupoid $X_W^{\Box}$ over ${\mathcal{C}}_L$ is pro-representable and the functor 
\[
(W_A, \iota_A, \alpha_A) \mapsto N_{W_A}
\]
induces an isomorphism between $X_W^{\Box}$ and $\widehat{\mathfrak{g}}$. 
\item The groupoid $X_{W, {\mathcal{P}}_W}^{\Box}$ over ${\mathcal{C}}_L$ is pro-representable and the functor 
\[
({\mathcal{P}}_{W, A}, W_A, \iota_A, \alpha_A) \mapsto (\alpha_A^{-1} ({\mathcal{P}}_{W, A}), N_{W_A})
\]
induces an isomorphism between $X_{W, {\mathcal{P}}_W}^{\Box}$ and $\widehat{\widetilde{\mathfrak{g}}}$. 
\end{enumerate}
\end{corollary}

\begin{proof}
In the case of $G = \GL_m$, we have $X_W^{\Box} = |X_W^{\Box}|$ and $X_{W, {\mathcal{P}}_W}^{\Box} = |X_{W, {\mathcal{P}}_W}^{\Box}|$ by the preceding argument of \cite[Corollary 3.6]{BHS17}. For a general $G$, considering the pushout $r_{\ast} (D)$ of $D$ by a faithful representation $G \to \GL_m$, we also have the similar two equations as above. Then the two pro-representabilities follows from Lemma \ref{additive_reconst} and Lemma \ref{DpdR_eqv}. 
\end{proof}

\subsection{Formally smoothness of the functor $W_{\dR}$}

Let $A$ be a finite dimensional ${\mathbb{Q}}_p$-algebra. For an object $\mathcal{M}$ of $\Phi \Gamma \left( {\mathcal{R}}_A \left[ \tfrac{1}{t} \right] \right)$ and a ${\mathcal{R}}_A$-lattice $D$ of $\mathcal{M}$ stable under $\varphi$, $\Gamma$, there exists covariant functors 
\begin{align*}
W_{\dR}^+ &\colon \Phi \Gamma ({\mathcal{R}}_A) \to \Rep_{A \otimes_{{\mathbb{Q}}_p} B_{\dR}^+} ({\mathcal{G}}_K) \quad D \mapsto W_{\dR}^+ (D) \\ 
W_{\dR} &\colon \Phi \Gamma ({\mathcal{R}}_A \left[ \tfrac{1}{t} \right]) \to \Rep_{A \otimes_{{\mathbb{Q}}_p} B_{\dR}} ({\mathcal{G}}_K) \quad \mathcal{M} \mapsto W_{\dR} (\mathcal{M}),  
\end{align*}
which are induced from the functors in the case of $A = {\mathbb{Q}}_p$ in \cite[Proposition 2.2.6 (2)]{Ber08}, constructed by \cite[Lemma 3.3.5]{BHS17}. 

Let $\mathcal{M}$ be an object of $G$-$\Phi \Gamma ({\mathcal{R}}_A \left[ \tfrac{1}{t} \right])$ and $D$ be an object of $G$-$\Phi \Gamma ({\mathcal{R}}_A)$ whose composition with the base change functor $\Phi \Gamma ({\mathcal{R}}_A) \to \Phi \Gamma ({\mathcal{R}}_A \left[ \tfrac{1}{t} \right])$ of $(\varphi, \Gamma)$-modules is isomorphic to $\mathcal{M}$. 
It follows that the functors 
\begin{align*}
W_{\dR}^+ &\colon G\text{-}\Phi \Gamma ({\mathcal{R}}_A) \to G\text{-}\Rep_{A \otimes_{{\mathbb{Q}}_p} B_{\dR}^+} ({\mathcal{G}}_K) \\ 
W_{\dR} &\colon G\text{-}\Phi \Gamma ({\mathcal{R}}_A \left[ \tfrac{1}{t} \right]) \to G\text{-}\Rep_{A \otimes_{{\mathbb{Q}}_p} B_{\dR}} ({\mathcal{G}}_K), 
\end{align*}
are induced by ones constructed above since $W_{\dR}^+$ is an exact tensor functor, that is seen by the construction. 
Since $B_{\dR}^+$ is a complete discrete valuation ring whose residue field is algebraically closed and $W_{\dR} (\mathcal{M})$ is defined over $B_{\dR}^+$ as $W_{\dR} = W_{\dR}^+ \otimes_{B_{\dR}^+} B_{\dR}$, the underlying torsors of the objects $W_{\dR}^+ (D)$ and $W_{\dR} (\mathcal{M})$ are trivial by \cite[Proposition 6.1.1]{Ces22}. 

We define a $(\varphi, \Gamma)$-module with $G$-structure $\mathcal{M}$ over ${\mathcal{R}}_A \left[ \tfrac{1}{t} \right]$, its \emph{triangulation} ${\mathcal{P}}_{\mathcal{M}}$ and its \emph{parameter}, and when $A = L$ the groupoid $X_{\mathcal{M}, {\mathcal{P}}_{\mathcal{M}}}$ of deformations over ${\mathcal{C}}_L$ in the same way as those over ${\mathcal{R}}_A$ for an affinoid algebra $A$ over ${\mathbb{Q}}_p$, especially in the case of $A$ is an object of ${\mathcal{C}}_L$. Note that a parameter of a fixed triangulation over ${\mathcal{R}}_A \left[ \tfrac{1}{t} \right]$ is not unique.

Let ${\mathcal{T}}_0$ be the Zariski open subspace of the rigid space $\mathcal{T}$, which is the complement of the characters $z^{{\mathbf{k}}, \epsilon \cdot z^{\mathbf{k}}}$ where $\mathbf{k} \in {\mathbb{Z}}^{\Sigma}$, and ${\mathcal{T}}_0^n$ be the Zariski open subspace of the rigid space ${\mathcal{T}}^d$, which is the complement of the characters $(\delta_1, \ldots, \delta_d)$ such that $\delta_i \delta_j^{-1} \in {\mathcal{T}}_0$ for $i \neq j$. For a fixed embedding $r \colon G \to \GL_m$ of regular type, we define ${\mathcal{T}}_{G, 0}$ be the inverse image of ${\mathcal{T}}_0^m \subset {\mathcal{T}}^m$ by the map ${\mathcal{T}}_G \to {\mathcal{T}}_{\GL_m} = {\mathcal{T}}^m$. 

Let $\mathcal{M}$ be a trianguline $(\varphi, \Gamma)$-module with $G$-structure over ${\mathcal{R}}_L \left[ \tfrac{1}{t} \right]$, ${\mathcal{P}}_{\mathcal{M}}$ be a triangulation of $\mathcal{M}$ with a parameter $\underline{\delta} \coloneqq (\delta_i)_{i \in [1, d]} \in (\delta_1, \ldots, \delta_d) \in {\mathcal{T}}^d$, $W \coloneqq W_{\dR} (\mathcal{M})$, ${\mathcal{P}}_W$ be the induced triangulation $W_{\dR} (\mathcal{M})$. 
Arguing as \cite[Corollary 3.3.9]{BHS17}, we obtain a morphism of groupoids over ${\mathcal{C}}_L$: 
\[
X_{\mathcal{M}, {\mathcal{P}}_{\mathcal{M}}} \to \widehat{\mathcal{T}}_{\underline{\delta}}^n \times_{\widehat{\mathfrak{t}}} X_{W, {\mathcal{P}}_W}. 
\]

\begin{proposition}
\label{GSp_formally_smoothness}
Assume the parameter $\underline{\delta}$ is locally algebraic and in ${\mathcal{T}}_{G, 0}$. Then the morphism 
\[
X_{\mathcal{M}, {\mathcal{P}}_{\mathcal{M}}} \to \widehat{\mathcal{T}}_{\underline{\delta}}^n \times_{\widehat{\mathfrak{t}}} X_{W, {\mathcal{P}}_W} 
\]
of groupoids over ${\mathcal{C}}_L$ is formally smooth. 
\end{proposition}

\begin{proof}
Let $A \twoheadrightarrow A'$ be a surjective map in ${\mathcal{C}}_L$, $x_{A'} = ({\mathcal{M}}_{A'}, {\mathcal{P}}_{{\mathcal{M}}_{A'}}, \pi_{A'})$ be an object of $X_{\mathcal{M}, {\mathcal{P}}_{\mathcal{M}}} (A')$, $y_{A'} = ({\underline{\delta}}_{A'}, W_{A'}, {\mathcal{P}}_{W_{A'}}, \iota_{A'})$ be its image in $(\widehat{\mathcal{T}}_{\underline{\delta}}^n \times_{\widehat{\mathfrak{t}}} X_{W, {\mathcal{P}}_W})(A')$. Let $y_A = (\underline{\delta}_A, W_A, {\mathcal{P}}_{W_A}, \iota_A)$ be an object of $(\widehat{\mathcal{T}}_{\underline{\delta}}^n \times_{\widehat{\mathfrak{t}}} X_{W, {\mathcal{P}}_W})(A)$ such that $\underline{\delta}_A = \underline{\delta}_B$ modulo $\ker (A \to A')$ and $A' \otimes_A (W_A, {\mathcal{P}}_{W_A}, \iota_A) = (W_{A'}, {\mathcal{P}}_{W_{A'}}, \iota_{A'})$. We prove that there exists an object $x_A = ({\mathcal{M}}_A, {\mathcal{P}}_{{\mathcal{M}}_A}, \pi_A)$ in $X_{\mathcal{M}, {\mathcal{P}}_{\mathcal{M}}} (A)$ whose image in $X_{\mathcal{M}, {\mathcal{P}}_{\mathcal{M}}} (A')$ is $x_{A'}$ and whose image in $(\widehat{\mathcal{T}}_{\underline{\delta}}^n \times_{\widehat{\mathfrak{t}}} X_{W, {\mathcal{P}}_W})(A)$ is isomorphic to $y_A$. 

Since the embedding $G \to \GL_m$ is of regular type, we use the notations and the constructions in Proposition \ref{reg_parab_ext}. 
By induction on $i \in [0, s-1]$, we will construct $(\varphi, \Gamma)$-modules with $L_{i+1} \cap B$-structures ${\mathcal{M}}_{A, i+1}$, which are the parabolic extensions of the $(\varphi, \Gamma)$-module with $L_i \cap B$-structure ${\mathcal{M}}_{A, i}$, which is already known by the assumption and compatible with isomorphisms $A' \otimes_A {\mathcal{F}}_{A, i} \cong {\mathcal{F}}_{A', i}$ and $W_{\dR} ({\mathcal{M}}_{A, i}) \cong W_{A, i}$, where $W_{A, i}$ is the $i$ successive parabolic extension of the $T$-torsor ${\mathcal{P}}_{W_{A'}} \times^B T$ given by the assumption. 

By the assumption of the induction, $i$ successive parabolic extensions of ${\mathcal{P}}_{W_A} \times^B T$, ${\mathcal{P}}_{W_{A'}} \times^B T$, ${\mathcal{P}}_{{\mathcal{M}}_A} \times^B T$ and ${\mathcal{P}}_{{\mathcal{M}}_{A'}} \times^B T$ are given by respectively $W_{A, i}$, $W_{A', i}$, ${\mathcal{M}}_{A, i}$ and ${\mathcal{M}}_{A', i}$. 
For each $i \in [1, s-1]$, there exist objects ${\mathcal{M}}_{\ad, A, i}$, ${\mathcal{M}}_{\ad, A', i}$, $W_{\ad, A, i}$ and $W_{\ad, A', i}$ in each category and the corresponding $(i+1)$-th extension classes are respectively written as follows: 
\begin{align*}
E_1 &\coloneqq \Ext_{\Rep_{A \otimes_{{\mathbb{Q}}_p} B_{\dR}} ({\mathcal{G}}_K)}^1 (A \otimes_{{\mathbb{Q}}_p} B_{\dR}, W_{\ad, A, i}), \\ 
E &\coloneqq \Ext_{\Rep_{A' \otimes_{{\mathbb{Q}}_p} B_{\dR}} ({\mathcal{G}}_K)}^1 (A' \otimes_{{\mathbb{Q}}_p} B_{\dR}, W_{\ad, A', i}), \\ 
E_2 &\coloneqq \Ext_{\Phi \Gamma \left( {\mathcal{R}}_{A'} \left[ \tfrac{1}{t} \right] \right)}^1 ({\mathcal{R}}_{A'} \left[ \tfrac{1}{t} \right], {\mathcal{M}}_{\ad, A', i}), 
\end{align*}
where $W_{\ad, A, i} \otimes_A A' \cong W_{\ad, A', i}$, ${\mathcal{M}}_{\ad, A, i} \otimes_A A' \cong {\mathcal{M}}_{\ad, A', i}$, $W_{\dR} ({\mathcal{M}}_{\ad, A, i}) \cong W_{\ad, A, i}$ and $W_{\dR} ({\mathcal{M}}_{\ad, A', i}) \cong W_{\ad, A', i}$ and the parameter of ${\mathcal{P}}_{{\mathcal{M}}_{A', i}}$ is in ${\mathcal{T}}_{G, 0} (A')$. 
There is the map 
\begin{align*}
\Ext_{\Phi \Gamma \left( {\mathcal{R}}_A \left[ \tfrac{1}{t} \right] \right)}^1 \left( {\mathcal{R}}_A \left[ \tfrac{1}{t} \right], {\mathcal{M}}_{\ad, A, i} \right) \to E_1 \times_E E_2 
\end{align*}
induced by $W_{\dR} \times (B \otimes_A 1)$. Then it suffices to prove the map is surjective and this follows by the argument of \cite[Theorem 3.4.4]{BHS17}. 
\end{proof}

\subsection{Locally irreducibility and accumulation property}

Let $\overline{\rho} \colon {\mathcal{G}}_K \to \GSp_{2n} (k_L)$ be a $\GSp_{2n}$-valued representation. We may similarly do the arguments of \cite[\S 3]{BHS17} in the case of $\GSp_{2n}$, using the results of \cite[\S 2]{BHS17} and Tannakian formalism. Especially, we note the $\GSp_{2n}$-variant of \cite[Corollary 3.7.10]{BHS17}.

\begin{proposition}
\label{locally_irreducibility}
Let $x = (\rho, \underline{\delta})$ be a point of $X_{\tri} (\overline{\rho})$ such that $D_{\rig} (\rho)$ is benign. Then the rigid space $X_{\tri} (\overline{\rho})$ is normal, irreducible and Cohen-Macaulay at $x$. 
\end{proposition}

\begin{proposition}
\label{omegaprime_flatness}
Let $x = (\rho, \underline{\delta}) \in X_{\tri} (\overline{\rho})$ be a crystalline point such that $D_{\rig} (r_{{\std}, {\ast}} (\rho))$ is $\varphi$-generic with regular Hodge-Tate type and assume that there exists a refinement of $D_{\crys} (D)$, then the morphism ${\omega}' \colon X_{\tri} (\overline{\rho}) \to {\mathcal{T}}_G$ is flat in a neighbourhood of $x$. 
\end{proposition}

\begin{proof}
This is shown by Theorem \ref{Xtri_mainpropoerty}, Proposition \ref{GSp_formally_smoothness}, \cite[Proposition 3.5.1, Corollary 3.5.9, Corollary 3.7.8]{BHS17} in the $G = \GSp_{2n}$ case, which are similarly shown with Tannakian formalism, and the proof of \cite[Proposition 4.1.1]{BHS17}.  
\end{proof}

Recall that the following definition (cf.\ \cite[Definition 2.11]{BHS15} or \cite[Definition 4.1.3]{BHS17}): 

\begin{definition}
Let $x = (\rho, \underline{\delta}) \in X_{\tri} (\overline{\rho})$ such that $\omega (x)$ is algebraic. We say that $X_{\tri} (\overline{\rho})$ satisfies the \emph{accumulation property} if the set of the points $(\rho', {\underline{\delta}}') \in X_{\tri} (\overline{\rho})$ such that $\rho'$ is benign accumulate at $x$ in $X_{\tri} (\overline{\rho})$ in the sense of \cite[\S 3.3.1]{BC09}. 
\end{definition}

\begin{proposition}
\label{Xtri_accumulation}
Let $x = (\rho, \underline{\delta}) \in X_{\tri} (\overline{\rho})$ be a crystalline point such that $D_{\rig} (r_{{\std}, {\ast}} (x))$ is $\varphi$-generic with regular Hodge-Tate type and assume that $\omega (x)$ is algebraic. Then $X_{\tri} (\overline{\rho})$ satisfies the accumulation property at $x$. 
\end{proposition}

\begin{proof}
We may show this as the proof of \cite[Proposition 4.1.4]{BHS17} from Lemma \ref{ext_saturated} and Proposition \ref{omegaprime_flatness}. 
\end{proof}

\section{Proof of the main theorem}

Let $G = \GSp_{2n}$ and $B$, $T$ be respectively the standard Borel subgroup and the maximal split torus of $\GSp_{2n}$. Let $\overline{\rho} \colon {\mathcal{G}}_K \to \GSp_{2n} (k_L)$ be a continuous $\GSp_{2n}$-representation. 
Let ${\mathbf{k}} \in ({\mathbb{Z}}^{\Sigma})^{2n}$ be a strict dominant regular Hodge-Tate type such that there is a character $\lambda \in X^{\ast} (T)$ over $L$ such that the element $r_{\std, \ast} (\lambda) \in X^{\ast} (T_{2n}) \cong ({\mathbb{Z}}^{\Sigma})^{2n}$ is ${\mathbf{k}}$. 

There is a quotient $R_{\overline{\rho}, \crys}^{\mathbf{k}}$ of $R_{\overline{\rho}}$, corresponding Zariski closed subspace ${\mathfrak{X}}_{\overline{\rho}, \crys}^{\mathbf{k}} \subset {\mathfrak{X}}_{\overline{\rho}}$ consists of the crystalline $\GSp_{2n}$-representations with Hodge-Tate type $\mathbf{k}$ by \cite[Proposition 3.0.9]{Bal12}. 

\begin{lemma}
\label{neighbor_benign}
Let $x$ be a point of ${\mathfrak{X}}_{\overline{\rho}, \crys}^{\mathbf{k}}$ and $U$ be an admissible open neighborhood of $x$. Then there exists a point $z \in U$ whose corresponding representation is benign. 
\end{lemma}

\begin{proof}
We will prove \'etale locally. There exists an \'etale morphism $U' \coloneqq \Sp (R) \to {\mathfrak{X}}_{\overline{\rho}, \crys}^{\mathbf{k}}$ whose image contains $x$ such that 
\[
D_{\crys} (V_R) \coloneqq ((R \widehat{\otimes}_{{\mathbb{Q}}_p} B_{\crys}) \otimes_R V_R)^{{\mathcal{G}}_K}
\]
is a finite free $K_0 \otimes_{{\mathbb{Q}}_p} R$-module and an isomorphism $D_{\dR} (V_R) \xrightarrow{\sim} K \otimes_{K_0} D_{\crys} (V_R)$ which are compatible with base changes, where $V_R$ is the restriction to $U'$ of the universal framed deformation of $\overline{\rho}$ by \cite[Corollary 6.3.3]{BC08}, \cite[Corollary 3.19]{Che10} and Tannakian formalism. 
So it suffices to construct a deformation of the $\GSp_{2n} (k_z)$-representation $\rho_z$ on $L$ such that $D_{\rig} (\rho_z)$ satisfies $\varphi$-genericity and noncriticality for any $z \in U \times_{{\mathfrak{X}}_{\overline{\rho}, \crys}^{\mathbf{k}}} U'$ because $R_{\overline{\rho}}^{\mathbf{k}} [1/p]$ is an integral domain, which follows by \cite[Proposition 4.1.5]{Bal12}. 
That is done by constructing an appropriate lift of $K$-filtered $\varphi$-module $D_{\crys} (\rho_z)$ over $L$ on $L[\epsilon] \in {\mathcal{C}}_L$ by Tannakian formalism and the construction after Lemma \ref{Berger_parameter} since weak admissibility of $K$-filtered $\varphi$-modules over $L$ is preserved by deformations on $A \in {\mathcal{C}}_L$, and that is done by the proof of \cite[Lemma 4.2]{Nak11}. 
\end{proof}

\begin{proposition}
\label{closure_irr_component}
The Zariski closure $\overline{{\mathfrak{X}}}_{\reg\text{-}\crys}$ of the following subset of ${\mathfrak{X}}_{\overline{\rho}}$: 
\[
{\mathfrak{X}}_{\reg\text{-}\crys} \coloneqq \{ x \in {\mathfrak{X}}_{\overline{\rho}} \ | \ \rho_x \text{is crystalline with regular Hodge-Tate type} \}
\]
is a union of irreducible components of ${\mathfrak{X}}_{\overline{\rho}}$. 
\end{proposition}

\begin{proof}
We may follow the proof of \cite[Proposition 5.10]{Iye19}. 
Let $\mathfrak{Z}$ be an irreducible component of $\overline{{\mathfrak{X}}}_{\reg\text{-}\crys}$. Endowing the closed subsets of ${\mathfrak{X}}_{\overline{\rho}}$ with the reduced closed structures (\cite[9.5.3, Proposition 4]{BGR}), the complement $\mathfrak{U} \coloneqq \mathfrak{Z} \backslash {\mathfrak{Z}}_{\sing}$ of the singular locus ${\mathfrak{Z}}_{\sing} \subset \mathfrak{Z}$ is an admissible open in $\mathfrak{Z}$. 

We may pick a smooth point $x \in \mathfrak{U}$ with the corresponding representation $\rho_x \colon {\mathcal{G}}_K \to \GSp_{2n} (k(x))$, and we may also assume that $\rho_x$ is benign by Lemma \ref{neighbor_benign}. 
To see that $x$ is also smooth, we have to show that $H^2 ({\mathcal{G}}_K, \ad (D_{\rig} (\rho_x))) = 0$. Then it suffices to show $H^2 (\ad (D_{\rig} (r_{\std, {\ast}} (\rho_x)))) = 0$ since $\ad (D_{\rig} (\rho_x))$ is a saturated $(\varphi, \Gamma)$-submodule of $\ad (D_{\rig} (r_{\std, \ast} (\rho_x)))$, and this follows from the argument in the first half part of the proof of \cite[Proposition 5.10]{Iye19} by the $\varphi$-genericity of $D_{\rig} (r_{\std, \ast} (\rho_x))$. 

The irreducible set $\mathfrak{Z}$ is contained in a unique irreducible component of ${\mathfrak{X}}_{\overline{\rho}}$. So the assertion is reduced to the surjectivity of the injection 
\[
T_x \mathfrak{Z} \hookrightarrow T_x {\mathfrak{X}}_{\overline{\rho}}
\]
of the tangent spaces. 

Recall that the triangulations of $D_{\rig} (\rho_x)$ are exactly parametrized by $w \in W(\GSp_{2n}, T)$ by Proposition \ref{G_Berger_corr}, and we write them as ${\mathcal{P}}_w$. Each pair $(x, {\mathcal{P}}_w)$ defines a point $y_w \in {\mathfrak{X}}_{\tri} (\overline{\rho})$. Replace $L$ with the field of definition $k(x)$ of $x$, which is a finite extension of $L$. 
Let ${\mathcal{P}}_w \left[ \tfrac{1}{t} \right]$ denote the induced triangulation of the $(\varphi, \Gamma)$-module with $\GSp_{2n}$-structure $D_{\rig} (\rho_x) \otimes_{{\mathcal{R}}_L} {\mathcal{R}}_L \left[ \tfrac{1}{t} \right]$ over ${\mathcal{R}}_L \left[ \tfrac{1}{t} \right]$ from the base change by ${\mathcal{R}}_L \to {\mathcal{R}}_L \left[ \tfrac{1}{t} \right]$. 

Write ${\mathcal{X}}_{\rho}$ for the usual functor of framed deformations over ${\mathcal{C}}_{L}$ of continuous $\rho \colon {\mathcal{G}}_K \to \GSp_{2n} (L)$, which is pro-represented by $R_{\rho}$. We define 
\begin{align*}
X_{D, \mathcal{M}} &\coloneqq X_D \times_{X_{D\left[ \frac{1}{t} \right]}} X_{D\left[ \frac{1}{t} \right], {\mathcal{M}}}, \\ 
{\mathcal{X}}_{D, \mathcal{M}} &\coloneqq X_{D, \mathcal{M}} \times_{X_D} {\mathcal{X}}_D
\end{align*}
as an $\GSp_{2n}$-analogue of the groupoids over ${\mathcal{C}}_L$ constructed in \cite[\S 3.5, \S 3.6]{BHS17} for a triangulation $\mathcal{M}$ of $D_{\rig} (\rho) \left[ \tfrac{1}{t} \right]$.
Note that there is the natural monomorphism of groupoids $X_{D, {\mathcal{P}}_w} \to X_{D, {\mathcal{P}}_w \left[ \frac{1}{t} \right]}$. 

Since $D_{\rig} (r_{\std, \ast} (\rho_x))$ is benign, it follows that ${{\widehat{\mathcal{O}}}_{X_{\tri} (\overline{\rho}), y_w}} \cong R_{{\mathcal{P}}_w \left[ \tfrac{1}{t} \right]}$ is an integral domain by arguing \cite[Theorem 3.6.2 (i), Corollary 3.7.8]{BHS17} (see also \cite[Proposition 5.7]{Iye19}) with Tannakian formalism and Proposition \ref{locally_irreducibility}.

It follows that there is a unique irreducible component ${\mathfrak{Y}}_w$ of ${\mathfrak{X}}_{\tri} (\overline{\rho})$, which contains $y_w$. Let ${\mathfrak{X}}_{\text{cr}}$ be the subset of points of $(\rho', \underline{\delta}') \in X_{\tri} (\overline{\rho})$ such that $\rho'$ is benign. By Proposition \ref{Xtri_accumulation}, there exists an affinoid neighborhood ${\mathfrak{U}}_w$ of $y_w$ such that ${\mathfrak{U}}_w \cap {\mathfrak{X}}_{\text{cr}}$is dense in ${\mathfrak{U}}_w$. Then ${\mathfrak{U}}_w \cap {\mathfrak{X}}_{\text{cr}} \cap {\mathfrak{Y}}_w$ is dense in ${\mathfrak{U}}_w \cap {\mathfrak{Y}}_w$, which is Zariski open and Zariski dense in ${\mathfrak{Y}}_w$, so ${\mathfrak{X}}_{\text{cr}} \cap {\mathfrak{Y}}_w$ is dense in ${\mathfrak{Y}}_w$. 
Then the image of ${\mathfrak{Y}}_w$ by the natural projection 
\[
X_{\tri} (\overline{\rho}) \to {\mathfrak{X}}_{\overline{\rho}} \times_L {\mathcal{T}}_{G, \reg} \to {\mathfrak{X}}_{\overline{\rho}}
\]
is $\mathfrak{Z}$. So we get the maps 
\[
\bigoplus_{w \in W(\GSp_{2n}, T)} T \left( {\mathcal{X}}_{\rho_x, {\mathcal{P}}_w \left[ \frac{1}{t} \right]} \right) = \bigoplus_{w \in W(\GSp_{2n}, T)} T_{y_w {\mathfrak{Y}}_w}\to T_x \mathfrak{Z} \hookrightarrow T_x \mathfrak{V} \cong T({\mathcal{X}}_{\rho})
\]
The map is compatible with the forgetful maps $X_{\rho_x, {\mathcal{P}}_w \left[ \frac{1}{t} \right]} \to X_{\rho_x}$. 
By the definition of ${\mathcal{X}}_{\rho_x, {\mathcal{P}}_w \left[ \frac{1}{t} \right]}$ and that ${\mathcal{X}}_{\rho_x} \to X_{\rho_x}$ is formally smooth, there exists a finite dimensional vector space $V$ over $L$ and the above sequence of the maps factors as 
\begin{align*}
\bigoplus_{w \in W(\GSp_{2n}, T)} T \left( {\mathcal{X}}_{\rho_x, {\mathcal{P}}_w \left[ \frac{1}{t} \right]} \right) &\cong 
\bigoplus_{w \in W(\GSp_{2n}, T)} \left( T \left( X_{\rho_x, {\mathcal{P}}_w \left[ \frac{1}{t} \right]} \right) \oplus V \right) \\
&\to 
T(X_{\rho_x}) \oplus V \cong T({\mathcal{X}}_{\rho_x})
\end{align*}
where the middle map is defined by $\sum_w T(\pi_w) \oplus {\text{id}}_V$ for each projection $\pi_w$ to the $w$-component. By Corollary \ref{surjectivity_GSp}, $\sum_w T(\pi_w)$ is surjective, so $T_x \mathfrak{Z} \to T_x \mathfrak{V}$ is an isomorphism. 
\end{proof}


Write $\det \colon \GSp_{2n} \to {\mathbb{G}}_m$ be the composition of $r_{\std}$ and the usual determinant map $\GL_{2n} \to {\mathbb{G}}_m$. Let $\omega \colon {\mathcal{G}}_K \to k_L$ be the mod $\varpi_L$-cyclotomic character. 
For an object $\overline{\rho}$ of $\GSp_{2n}$-$\Rep_{k_L} ({\mathcal{G}}_K)$ whose underlying torsor is trivial, we write $\ad(\overline{\rho}) \coloneqq \Lie_{k_L} (\Aut_G (\overline{\rho}))$, $\ad^0 (\overline{\rho})$ be the kernel of the map $\ad(\overline{\rho}) \to {\mathbb{G}}_{a, k_L}, \ x \to (\det (1+x) - 1)/\epsilon$ and $\chi_{\text{LT}} \colon {\mathcal{G}}_K \to K^{\ast}$ be the Lubin-Tate character corresponding to the fixed uniformizer $\varpi_K \in {\mathcal{O}}_K$. We naturally identify continuous representations ${\mathcal{G}}_K \to \GSp_{2n} (A)$ with objects of $\GSp_{2n}$-$\Rep_A ({\mathcal{G}}_K)$ whose underlying torsor is trivial for $A \in {\mathcal{C}}_{{\mathcal{O}}_L}$. 

\begin{theorem}
\label{gsp_density}
Assume that $\overline{\rho}$ satisfies the following conditions: 
\begin{enumerate}
\item $H^0 ({\mathcal{G}}_K, \ad (\overline{\rho})) \cong k_L$
\item ${\mathfrak{X}}_{\reg\text{-}\crys}$ is non-empty, 
\item $H^0 ({\mathcal{G}}_K, \ad(\overline{\rho}) \otimes \omega) = 0$, 
\item $\zeta_p \notin K^{\ast}$, or $p \nmid 2n$. 
\end{enumerate}
Then we have an equality $\overline{{\mathfrak{X}}}_{\reg\text{-}\crys} = {\mathfrak{X}}_{\overline{\rho}}$. 
\end{theorem}

\begin{proof}
First, if $\zeta_p \notin K^{\ast}$, we have $H^2 ({\mathcal{G}}_K, \ad(\overline{\rho})) = 0$ by the assumptions, the following exact sequence 
\[
H^2 ({\mathcal{G}}_K, \ad^0 (\overline{\rho})) \to H^2 ({\mathcal{G}}_K, \ad(\overline{\rho})) \to H^2 ({\mathcal{G}}_K, k_L) \to 0
\]
and the Tate duality. So it follows that ${\mathfrak{X}}_{\overline{\rho}}$ is irreducible and then the theorem follows by Proposition \ref{closure_irr_component}. 
Next, we prove the theorem when $p \nmid 2n$. Let $\alpha \in {\mathbb{Z}}_{\geq 0}$ be a maximal integer such that $\zeta_{p^{\alpha}} \in K$ and $\mu \coloneqq \mu_{p^{\infty}} (K) = \mu_{p^{\alpha}} (K)$ be the group of the $p$-power roots of unity in $K$. There is the decomposition 
\[
R_{\overline{\rho}} \left[ \tfrac{1}{p} \right] \cong \prod_{\chi \colon \mu \to {\mathcal{O}}_L^{\ast}} R_{\overline{\rho}}^{\chi} \left[ \tfrac{1}{p} \right], 
\]
where $R_{\overline{\rho}}^{\chi} = R_{\overline{\rho}} \otimes_{{\mathcal{O}}_L [\mu], \ \chi} {\mathcal{O}}_L$, by the argument of \cite[\S 3]{BIP22}. It induces the equality 
\[
{\mathfrak{X}}_{\overline{\rho}} = \bigsqcup_{\chi \colon \mu \to {\mathcal{O}}_L^{\ast}} {\mathfrak{X}}_{\overline{\rho}}^{\chi}
\]
of rigid spaces. It follows that the rigid spaces ${\mathfrak{X}}_{\overline{\rho}}^{\chi}$ for all character $\chi \colon \mu \to {\mathcal{O}}_L^{\ast}$ are isomorphic by the assumptions and the argument in the last half of the proof of \cite[Theorem 4.4]{Nak11}. Then it suffices to show that a ${\mathfrak{X}}_{\overline{\rho}}^{\chi}$ for some $\chi$ is irreducible by Proposition \ref{closure_irr_component}. 

We follow the argument of the proof of \cite[Theorem 4.16]{Nak10}. Let $A$ be a object of ${\mathcal{C}}_{{\mathcal{O}}_L}$ and $I \subset A$ be a non-zero ideal such that $I {\mathfrak{m}}_A = 0$ and $D_{\overline{\rho}}^{\chi}$ is the framed deformation functor pro-represented by $R_{\overline{\rho}}$. We will prove that for each $[\rho_{A/I}] \in D_{\overline{\rho}}^{\chi} (A/I)$ lifts to $D_{\overline{\rho}}^{\chi} (A)$. We represent $\rho_{A/I}$ by a continuous homomorphism ${\mathcal{G}}_K \to \GSp_{2n} (A/I)$, for which we also write $\rho_{A/I}$, by fixing an $A/I$-basis. 

Since $R_{\overline{\rho}}$ is formally smooth for example by \cite[Lemma 4.1]{BIP21}, we can take a continuous character $c_A \colon {\mathcal{G}}_K^{\text{ab}} \to A^{\ast}$ which is a lift of $\rho_{A/I}$ such that $c_A (\rec_K (\zeta_{p^n})) = \chi (\zeta_{p^n})$. Then taking a continuous lift ${\widetilde{\rho}}_A \colon {\mathcal{G}}_K \to \GSp_{2n} (A)$ of $\rho_{A/I}$ such that $\det ({\widetilde{\rho}}_A (g)) = c_A (g)$ for any $g \in {\mathcal{G}}_K$, ${\widetilde{\rho}}_A (g_1 g_2) {\widetilde{\rho}}_A (g_2)^{-1} {\widetilde{\rho}}_A (g_1)^{-1}$ defines the $2$-cocycle $f \colon {\mathcal{G}}_K \times {\mathcal{G}}_K \to I \otimes_{k_L} \ad (\overline{\rho})$. It is contained in $I \otimes_{k_L} \ad^0 (\overline{\rho})$ by $\det ({\widetilde{\rho}}_A) = c_A$ is an homomorphism. By the assumption (1), we have $H^2 ({\mathcal{G}}_K, \ad^0 (\overline{\rho})) = 0$. 
Then it follows that there is a cochain $d \colon {\mathcal{G}}_K \to I \otimes_{k_L} \ad^0 (\overline{\rho})$ such that ${\widetilde{\rho}}_A \otimes (\text{id} + d)$ is a continuous homomorphism $\rho_A \colon {\mathcal{G}}_K \to \GSp_{2n} (A)$ such that $\rho_A$ is a lift of $\rho_{A/I}$ and $\det (\rho_A) = c_A$. 
\end{proof}







\end{document}